\documentclass[12pt,letterpaper]{scrartcl}

\usepackage{amsmath}
\usepackage{amsthm}
\usepackage{amssymb}
\usepackage{stmaryrd}
\usepackage{accents}
\usepackage{bbm}
\usepackage[mathscr]{euscript}
\usepackage{xcolor}
\definecolor{Myblue}{rgb}{0,0,0.6}  
\usepackage[colorlinks,citecolor=Myblue,linkcolor=Myblue,urlcolor=Myblue,pdfpagemode=None]{hyperref}
\usepackage{subcaption}
\usepackage{graphicx}
\usepackage{epstopdf}
\usepackage[totalwidth=420pt, totalheight=590pt]{geometry}
\usepackage[inline,shortlabels]{enumitem}
\usepackage[nottoc]{tocbibind}
\usepackage{hyperref}
\usepackage{fancyhdr}
\usepackage{xcolor}
\usepackage{tikz}
\usetikzlibrary{decorations.markings}
\usetikzlibrary{knots}
\usetikzlibrary{arrows}
\usetikzlibrary{cd}

\allowdisplaybreaks

\theoremstyle{definition}
\newtheorem{defn}{Definition}
\newtheorem{thm}[defn]{Theorem}
\newtheorem{prp}[defn]{Proposition}
\newtheorem{lem}[defn]{Lemma}

\newtheorem{rem}[defn]{Remark}

\numberwithin{equation}{section}
\numberwithin{defn}{section}
\numberwithin{figure}{section}

\newcommand{\pic}[2][0.75]{
	\begin{tikzpicture}[scale=0.5,baseline={([yshift=-.5ex]current bounding box.center)}]
	\node at (0,0) {\includegraphics[scale=#1]{figures/#2}};
	\end{tikzpicture}
}

\begin{document}
\def\it{\textit}
\def\mcA{\mathcal{A}}
\def\mcB{\mathcal{B}}
\def\mcC{\mathcal{C}}
\def\mcD{\mathcal{D}}
\def\mcI{\mathcal{I}}
\def\mcS{\mathcal{S}}
\def\mcZ{\mathcal{Z}}
\def\mcACA{{_A \mcC_A}}
\def\mcD{\mathcal{D}}
\def\euT{\mathscr{T}}
\def\opk{\Bbbk}
\def\opA{\mathbb{A}}
\def\opC{\mathbb{C}}
\def\opR{\mathbb{R}}
\def\opZ{\mathbb{Z}}
\def\opid{\mathbbm{1}}
\def\a{\alpha}
\def\b{\beta}
\def\g{\gamma}
\def\d{\delta}
\def\D{\Delta}
\def\vareps{\varepsilon}
\def\l{\lambda}
\def\abar{\overline{\a}}
\def\bbar{\overline{\b}}
\def\gbar{\overline{\g}}
\def\ddbar{\overline{\d}}
\def\mubar{\overline{\mu}}
\def\pibar{{\bar{\pi}}}

\def\opp{{\operatorname{op}}}
\def\id{\operatorname{id}}
\def\im{\operatorname{im}}
\def\Hom{\operatorname{Hom}}
\def\End{\operatorname{End}}
\def\tr{\operatorname{tr}}
\def\ev{\operatorname{ev}}
\def\coev{\operatorname{coev}}
\def\evt{\widetilde{\operatorname{ev}}}
\def\coevt{\widetilde{\operatorname{coev}}}
\def\Id{\operatorname{Id}}
\def\Vect{\operatorname{Vect}}
\def\loc{{\operatorname{loc}}}
\def\orb{{\operatorname{orb}}}
\def\coker{{\operatorname{coker}}}
\def\Ind{{\operatorname{Ind}}}
\def\Irr{{\operatorname{Irr}}}
\def\Dim{\operatorname{Dim}}

\def\Lra{\Leftrightarrow}
\def\Ra{\Rightarrow}
\def\ra{\rightarrow}
\def\lra{\leftrightarrow}
\def\xra{\xrightarrow}

\def\varone{\iota}

\newcommand{\eqrefO}[1]{\hyperref[eq:O#1]{\text{(O#1)}}}
\newcommand{\eqrefT}[1]{\hyperref[eq:T#1]{\text{(T#1)}}}

\title{Fibonacci-type orbifold data\\ in Ising modular categories}

\author{
Vincentas Mulevi\v{c}ius \quad
Ingo Runkel\\[0.5cm]
\normalsize{\texttt{\href{mailto:vincentas.mulevicius@uni-hamburg.de}{vincentas.mulevicius@uni-hamburg.de}}} \quad
\normalsize{\texttt{\href{mailto:ingo.runkel@uni-hamburg.de}{ingo.runkel@uni-hamburg.de}}}\\[0.1cm]
{\normalsize\slshape Fachbereich Mathematik, Universit\"{a}t Hamburg, Germany}\\[-0.1cm]
}

\date{}
\maketitle

\begin{abstract}
    An orbifold datum is a collection $\opA$ of algebraic data in a modular fusion category $\mcC$. It allows one to define a new modular fusion category $\mcC_\opA$ in a construction that is a generalisation of taking the Drinfeld centre of a fusion category. Under certain simplifying assumptions we characterise orbifold data $\opA$ in terms of scalars satisfying polynomial equations and give an explicit expression which computes the number of
    isomorphism classes of simple objects in $\mcC_\opA$.
    
In Ising-type modular categories we find new examples of orbifold data which -- in an appropriate sense -- exhibit Fibonacci fusion rules. The corresponding orbifold modular categories have 11 simple objects, and for a certain choice of parameters one obtains the modular category for $sl(2)$ at level 10. This construction inverts the extension of the latter category by the $E_6$ commutative algebra.
\end{abstract}
\newpage

\setcounter{tocdepth}{2}

\tableofcontents

\section{Introduction and summary}

The notion of an orbifold datum was introduced in \cite{CRS1} to describe a generalised notion of ``orbifolding'' a topological quantum field theory by carrying out an internal state sum construction, where the underlying manifold is stratified by a defect foam. For the three-dimensional Reshetikhin-Turaev TQFT defined by a modular fusion category $\mcC$, an orbifold datum $\opA$ is described algebraically as a tuple \cite{CRS3}
$$
\opA ~=~ \big(\,A , T, \alpha, \bar\alpha , \psi , \phi \,\big) \ ,
$$
where $A$ is a ($\Delta$-separable, symmetric) Frobenius algebra in $\mcC$, $T$ is an $A$-$(A\otimes A)$-bimodule in $\mcC$, $\alpha, \bar\alpha$ are endomorphisms of $T \otimes T$, $\psi$ is an invertible $A$-$A$-bimodule endomorphism of $A$, and finally $\phi$ is a constant. The conditions $\opA$ has to satisfy -- as well as their interpretation in a TQFT with point, line and surface defects -- can be found in \cite{CRS3}.

Given an orbifold datum $\opA$ in a modular fusion category $\mcC$, the construction in \cite{CRS3,Carqueville:2020} defines a new TQFT -- the generalised orbifold -- by evaluating the TQFT for $\mcC$ on stratified manifolds where 2-strata are labelled by $A$, 1-strata by $T$ and 0-strata by $\alpha$ or $\bar\alpha$.
On the other hand, the pair $(\mcC,\opA)$ gives rise to a new modular fusion category $\mcC_{\opA}$ \cite{Mulevicius:2020bat} (under the extra assumption that $\opA$ is simple,
    see below), 
and one finds that
the Reshetikhin-Turaev TQFT for $\mcC_{\opA}$ is equivalent to this generalised orbifold  \cite{Carqueville:2020}.

All previously known examples of orbifold data were of three types: a) those obtained from a commutative algebra $A$, b) those obtained from a $G$-crossed ribbon category for a finite group $G$, and c) those obtained from a spherical fusion category \cite{CRS3}.
In this paper we provide an explicit example which is not of these three types, and we believe that more examples can be found with the methods we use. 
As outlined below, two motivations to look for such more general examples are provided by the classification of modular fusion categories and by topological phases of matter.

\medskip

We now describe the construction and the example in more detail. 

\medskip

Let $\mcC$ be a modular fusion category. Our first (of three) simplifying assumption is that the fusion rules $N_{ij}^{~k}$ of $\mcC$ are all either $0$ or $1$.

The second (and most drastic) simplifying assumption is that the algebra $A$ is a direct sum of copies of the tensor unit $\opid$ of $\mcC$, itself thought of as a commutative Frobenius algebra. 
That is, $A = \bigoplus_{a \in B} \opid_a$, where $a \in B$ indexes the different copies of $\opid$. This is definitely not the most general ansatz, but it will be good enough to find new examples.

Accordingly, $T = \bigoplus_{a,b,c \in B} {}_at_{bc}$, where ${}_at_{bc} \in \mcC$. The third simplifying assumption we make is that ${}_at_{bc}$ is either zero or a simple object. 

Under these three assumptions, in Section~\ref{sec:poly} we state the conditions an orbifold datum has to satisfy as polynomial equations for the expansion coefficients in an appropriate basis. This is very similar to expressing e.g.\ the pentagon equation for the associator of a fusion category in terms of $F$-matrices as ``$FF = \sum FFF$''. And indeed, this equation appears in the special case that $\mcC$ is the category of finite-dimensional complex vector spaces.

Given a list of scalars solving the polynomial equations characterising an orbifold datum $\opA$ as above, one can find expressions in terms of these scalars that compute whether $\opA$ is simple (Proposition~\ref{prop:A-simple}), as well as the number of simple objects in $\mcC_\opA$ (Proposition~\ref{prop:count-simple}).

\medskip

Our example is an orbifold datum in Ising-type modular fusion categories $\mcI_{\zeta,\epsilon}$. These are precisely the modular fusion categories with three simple objects $\opid$, $\sigma$, $\vareps$ and fusion rules 
$$
\vareps \otimes \vareps \,\cong\, \opid
\qquad , \qquad  \sigma \otimes \sigma \,\cong\, \opid \,\oplus\, \vareps \ . 
$$
There are 16 such categories, indexed by $\zeta \in \opC$ with $\zeta^8 = -1$ (equivalently, $\zeta$ is a primitive 16${}^\text{th}$ root of unity), and by a sign $\epsilon \in \{ \pm 1\}$, see \cite[App.\,B]{DGNO}. We describe $\mcI_{\zeta,\epsilon}$ in detail in Section~\ref{sec:Ising}.

We make the ansatz that $A$ is a direct sum of two copies of $\opid$, which we index by $\varone$ and $\varphi$:
\begin{equation}
A = \opid_{\varone} \oplus \opid_\varphi~.
\end{equation}
For ${}_at_{bc}$ we make the ansatz
\begin{equation}\label{eq:t-ansatz}
	{}_at_{bc} = \begin{cases} \opid &;~ \text{either 0 or 2 of $a,b,c$ are $\varphi$}
	\\
	\sigma &;~ \text{all of $a,b,c$ are $\varphi$}
	\\
0 &;~ \text{else}
	\end{cases}
\end{equation}
This ansatz may seem completely ad hoc, but we explain in Remark~\ref{rem:E6} below how to arrive at these conditions.
Note that if we define numbers $M_{bc}^{~a}$ to be 1 if ${}_at_{bc}\neq 0$ and 0 else, we obtain the Fibonacci fusion rules. Hence the title of this paper. In Section~\ref{sec:orbdata} we solve the resulting polynomial equations (up to a simple gauge freedom), and the solutions are (Theorem~\ref{thm:main-detail}):

\begin{thm}\label{thm:main}
For each primitive $48^\text{th}$ root of unity $h \in \opC$ and for each $\epsilon \in \{\pm 1\}$ we obtain an orbifold datum $\opA_{h,\epsilon}$ in $\mcI_{\zeta,\epsilon}$, where $\zeta = h^3$ and where $A$ and $T$ are as above.
\end{thm}

There are 16 possible choices for $h$ and thus 32 orbifold modular categories $(\mcI_{h^3,\epsilon})_{\opA_{h,\epsilon}}$ which turn out to be pairwise non-equivalent (Proposition~\ref{prop:all-different}).
Each category $(\mcI_{h^3,\epsilon})_{\opA_{h,\epsilon}}$ has 11 isomorphism classes of simple objects, and its global dimension is given by (Proposition~\ref{prop:C_A-properties})
\begin{equation}
\Dim\!	\big( (\mcI_{h^3,\epsilon})_{\opA_{h,\epsilon}} \big)
~=~ 24\, \big(h^2+h^{-2}\big)^{-2} \ ,
\end{equation}
which takes values $6.43..$ and $89.5..$ depending on $h$. 
Objects of $\mcC_\opA$ consist of $A$-$A$-bimodules together with certain extra data.
In the case of $(\mcI_{h^3,\epsilon})_{\opA_{h,\epsilon}}$, we compute  the underlying bimodules (and hence objects in $\mcI_{\zeta,\epsilon}$) for the 11 simple objects, see Section~\ref{sec:adjoint}. 
This allows one for example to determine their quantum dimensions.
The next remark suggests that the $(\mcI_{h^3,\epsilon})_{\opA_{h,\epsilon}}$ should in an appropriate sense all be Galois conjugates of $\mathcal{C}(sl(2),10)$. 

\begin{rem}\label{rem:E6}~
\begin{enumerate}
	\item
Denote by $\mathcal{C}(sl(2),10)$ the modular fusion category of integrable highest weight
representations of the affine Lie algebra $\widehat{sl}(2)_{10}$. 	
We arrived at the above ansatz for $A$ and $t$ by investigating the extension of $\mathcal{C}(sl(2),10)$ by the commutative algebra $A=\underline{0} \oplus \underline{6}$, where the underlined number denotes the Dynkin label.\footnote{
	{} We would like to thank Terry Gannon for pointing
	out the relevance of this example to our attempts to find orbifold data not in one of the three classes mentioned above.}
From conformal embeddings (see e.g.\ \cite[Ch.\,17.5]{DiFr} and \cite{Ostrik:2001}), this extension is known to be $\mathcal{C}(sp(4),1)$ which is an Ising-type category (see \cite[Sec.\,6.4]{Davydov:2010}). The algebra $A$ also describes the $E_6$ modular invariant of the $SU(2)$-WZW model at level $10$. 

As mentioned above, extensions are an instance of a generalised orbifold. Conjecturally, the process of taking a generalised orbifold can be inverted by another generalised orbifold,
so it was expected that there is an orbifold datum in this Ising-type category which inverts the extension. A general discussion of inverse orbifolds for extensions will be given in \cite{Mu-prep} and leads to the above ansatz.

We note that for two-dimensional topological field theories, the procedure of passing to a generalised orbifold is indeed invertible \cite[Sec.\,4]{Brunner:2014lua}.

\item From part 1 (and \cite{Mu-prep}), 
one can conclude that for the values $h,\epsilon$ that correspond to the inverse of the extension from
$\mathcal{C}(sl(2),10)$ to $\mathcal{C}(sp(4),1)$ (which are $h=\exp(\pi i \frac{19}{24})$ and $\epsilon=-1$, see Remark~\ref{rem:which-h}), we get $(\mcI_{h^3,\epsilon})_{\opA_{h,\epsilon}} \cong \mathcal{C}(sl(2),10)$.

\item As will be explained in \cite{Mu-prep}, the fact that set $B = \{\varone,\varphi\}$ in our example has two elements is related to the observation that the $11{\times}11$ matrix $Z_{E_6}$ for the $E_6$ modular invariant satisfies $Z_{E_6} Z_{E_6} = 2\,Z_{E_6}$. 
Analogously, the $29{\times}29$ matrix $Z_{E_8}$ for the $E_8$ modular invariant  of the $SU(2)$-WZW model at level 28 satisfies $Z_{E_8} Z_{E_8} = 4\, Z_{E_8}$. 
Thus the inverse orbifold to the  extension from $\mcC(sl(2),28)$ to $\mcC(G_2,1)$ (which is a Fibonacci category) will be described by an orbifold datum with $A = \opid^{\oplus 4}$ in this Fibonacci category. 

Another direction in which one could attempt to generalise the example from Theorem~\ref{thm:main} is to replace the Ising-type categories by more general Tambara-Yamagami categories.

We hope to investigate these examples in more detail in the future.
\end{enumerate}

\end{rem}

\begin{rem}\label{rem:enriched}
The notion of an orbifold datum is very similar to that of a monoidal category enriched over a braided category \cite{Morrison:2017}, and, equivalently, to module tensor categories \cite{Morrison:2018}. The latter also appear in the description of so-called anchored planar algebras \cite{Henriques:2016}. 
For instance, the example in Theorem~\ref{thm:main} can be thought of as a category with two simple objects $\varone$ and $\varphi$, enriched in $\mcI(\zeta,\epsilon)$, such that $\varone$ is the tensor unit and e.g.\ the Hom-object $\varphi \otimes \varphi \to \varphi$ is given by $\sigma \in \mcI(\zeta,\epsilon)$. In the context of anchored planar algebras, the corresponding example is treated in \cite[Ex.\,3.15]{Henriques:2016}.
We also note that
via the connection to enriched monoidal categories, the construction of $\mcC_{\opA}$ should correspond to the enriched centre of \cite{KZ}.

We will not pursue this point of view in the present paper but hope to return to it in the future.
\end{rem}

To conclude the introduction, let use give two more motivations for the work presented in this paper.

\medskip

The first motivation stems from the classification of modular fusion categories.
There is, up to equivalence, only a finite number of modular fusion categories with a given number of simple objects \cite{Bruillard:2015kdw}, a property called {\em rank-finiteness}. It therefore makes sense to classify such categories by number of simple objects (the rank). This has been done up to rank 5 \cite{Rowell:2007dge,Bruillard:2015}, and rank 6 is in progress \cite{Green:2019,Creamer:2019}.

Up to these ranks, the classification produces only group- and Lie-theoretic examples. A systematic approach to produce more exotic examples of modular categories is to consider Drinfeld doubles of (the fusion categories associated to) finite index subfactors, see e.g.\ \cite{Hong:2007,Evans:2010yr,Jones:2013,Gannon:2016bch} for more details and references. The smallest and most famous exotic example is the Haagerup subfactor whose fusion category has rank 6, and whose Drinfeld double has rank 12. This example is currently very much out of reach of the direct classification of modular categories by rank.

Drinfeld doubles of fusion categories are by definition examples of modular fusion categories of trivial Witt class \cite{Davydov:2010}. One can therefore think of the systematic study of fusion categories as the exploration of the trivial Witt class of modular categories. 
On the other hand, passing from $\mcC$ to $\mcC_\opA$ does -- conjecturally -- stay within the same Witt class.
The problem we would like to advocate with the simple example considered in this paper is:
\begin{quote}
	Try to explore non-trivial Witt classes of modular categories by systematically studying orbifold data $\opA$ in a given small representative of a Witt class, e.g.\ by solving the simplified equations in Section~\ref{sec:poly} and then computing $\mcC_{\opA}$.
\end{quote}

\medskip

The second motivation stems from the study of topological phases of matter.
Namely,
unitary modular categories $\mcC$ model anyons in two-dimensional topological phases of matter, see e.g.\ \cite{Rowell:2017}. If such a topological phase has a finite symmetry group $G$, one can try to gauge this symmetry. If this is unobstructed, one arrives at a new topological phase described by the unitary modular category $\mcC_G^{\times,G}$ \cite{Barkeshli:2014cna,CGPW}. This is (conjecturally) a special case of the generalised orbifold construction above \cite[Sec.\,1.1]{Mulevicius:2020bat}. In this sense one can think of an orbifold datum as a generalised notion of symmetry for a topological phase of matter. In another approach \cite{CZW}, Hopf monads are used to describe generalised symmetries.

In fact, by Remark~\ref{rem:E6}
the generalised orbifold (or gauging of the ``generalised symmetry'') by one of the specific examples we describe in this paper is inverse to the anyon condensation of the $E_6$ commutative algebra in the chiral $su(2)_{10}$ WZW-model. See \cite{Kong:2013aya} and references therein for more details on anyon condensation.

Since the Ising category does not support universal quantum computation, but $\mcC(sl(2),10)$ does (see \cite{Naidu:2009,Rowell:2017}), it would be interesting to see if the fact that a generalised orbifold can turn the former into (a close relative of) the latter has applications in topological quantum computation.

Finally, assuming the validity of the relation in Remark~\ref{rem:enriched}, an orbifold datum in the modular fusion category given by representations of a rational vertex operator algebra describes a gapless edge of a 2d topological order \cite{Kong:2019byq,Kong:2019cuu}.

\bigskip

This paper is organised as follows. In Section~\ref{sec:poly} we describe our simplifying assumptions for an orbifold datum $\opA$ and derive the polynomial equations it has to satisfy. In Section~\ref{sec:CA-properties} we study the corresponding orbifold modular category $\mcC_\opA$ and give explicit expressions for its global dimension and number of simple objects. In Section~\ref{sec:Ising} we review Ising-type modular categories, and in Section~\ref{sec:orbdata} we classify solutions for Fibonacci-type orbifold data in these categories. In Section~\ref{sec:adjoint} we illustrate how to gain information about objects of the orbifold modular category form their underlying bimodules. An appendix contains some of the more technical calculations.

\subsubsection*{Acknowledgements}

We would like to thank Terry Gannon for suggesting the $E_6$-extension of $\mcC(sl(2),10)$ as a place to look for a new example of orbifold data.
We are grateful to Nils Carqueville and Zhenghan Wang for helpful discussions and comments on a draft of this article.
VM was supported by the Research Training Group RTG\,1670 of the Deutsche Forschungsgemeinschaft.
IR acknowledges support from the RTG\,1670 and the Cluster of Excellence EXC 2121.

\section{Orbifold data via polynomial equations}\label{sec:poly}

In this section we review the notion of an orbifold datum in a modular fusion category
(i.e.\ a semisimple modular tensor category). Under a number of simplifying assumptions, we give a description of an orbifold datum in terms of polynomial equations, much like the pentagon equation for the associator of a fusion category 

\medskip

Let $\mcC$ be a modular fusion category
over an algebraically closed field $\opk$.
We fix a set $I$ of representatives of its isomorphism classes of
simple objects so that $\opid\in I$.
For $i,j,k\in I$, we denote by $N_{ij}^k$ the dimension of the Hom space $\mcC(i\otimes j, k)$.
For the rest of the section we make the
\begin{equation}
\tag{A1}
\label{eq:assumption1}
\boxed{~\text{assumption:} \qquad N_{ij}^k \in \{0,1\} ~.~ }
\end{equation}
It is also going to be useful to fix a nonzero element 
(i.e.\ a basis) of those spaces $\mcC(i\otimes j, k)$ which are one-dimensional.
We will denote this element, as well as its dual in $\mcC(k, i\otimes j)$ with respect to the composition pairing by
\begin{equation}
\lambda_{(ij)k} = \pic[2.0]{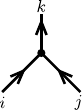} ~,
\qquad
\lambda^{(ij)k} = \pic[2.0]{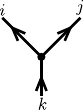} ~,
\end{equation}
i.e.\ they satisfy $\lambda_{(ij)k} \circ \lambda^{(ij)k} = \id_k$.
We follow the conventions of \cite{FRS1}, in particular the diagrams are to be read from bottom to top.
If $i=\opid$ and $k=j$ we choose $\lambda_{(\opid j)j}$ to be the left unitor $\opid \otimes j \to j$ of $\mcC$. In the same way, we choose $\lambda_{(i \opid)i}$ to be the right unitor.

\medskip

The associator morphisms are encoded in the $F$-matrix and its inverse $G$, whose elements are indexed by $i,j,k,l,p,q \in I$.
They are defined by the following relations:
\begin{equation}
\pic[2.0]{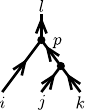} = \sum_{q\in I} F^{(ijk)l}_{pq} \pic[2.0]{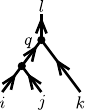}, \quad
\pic[2.0]{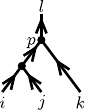} = \sum_{q\in I} G^{(ijk)l}_{pq} \pic[2.0]{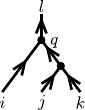}.
\end{equation}
Since $\lambda_{(\opid x)x}$ and $\lambda_{(x \opid)x}$ are unitors, if at least one of $i,j,k$ is $\opid$ and the corresponding $F$-matrix element is not automatically zero by the fusion rules, we have $F^{(ijk)l}_{pq} = 1$.

Similarly, the braiding morphisms are given by the $R$-matrix and its inverse, whose elements are, for $i,j,k\in I$:
\begin{equation}
\pic[2.0]{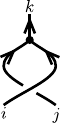} = R^{(ij)k} \pic[2.0]{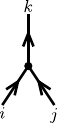}, \quad
\pic[2.0]{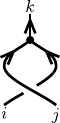} = R^{-(ij)k} \pic[2.0]{2_R_rhs.pdf}.
\end{equation}

An orbifold datum in $\mcC$ \cite{CRS1,CRS3} is a tuple $\opA = (A, T, \a, \abar, \psi, \phi)$.
Here, $A$ is a symmetric $\D$-separable Frobenius algebra in $\mcC$, and $T$ is an $A$-$A\otimes A$-bimodule. We will index the first and second tensor factor in $A \otimes A$ by $1$ and $2$, respectively, to distinguish between the two $A$ actions on $T$ from the right. Then $\a$ and $\abar$ are $A$-$A \otimes A \otimes A$-bimodule isomorphisms
\begin{equation}
\a : T \otimes_2 T \ra T \otimes_1 T, \qquad
\abar: T \otimes_1 T \ra T \otimes_2 T ~,
\end{equation}
where $i=1,2$ denotes the tensor product over $A$ with respect to the indicated action. Finally, $\psi: A \ra A$ is an $A$-$A$-bimodule isomorphism and $\phi \in \opk^\times$ is a number.
The constituents of an orbifold datum must fulfil the conditions (O1)--(O8) listed in \cite[Sec.\,2.3]{Mulevicius:2020bat}. These conditions first appeared in~\cite{CRS1,CRS3} under the name special orbifold datum.
For example, conditions (O2) and (O3) state that $\a$ and $\abar$ are inverse to each other up to the action of $\psi$'s.

\medskip

The main goal of this paper is to construct an example of an orbifold datum.
To this end, in addition to \eqref{eq:assumption1} we make another
\begin{equation}
\tag{A2}
\label{eq:assumption2}
\boxed{~\text{assumption:} \quad A = \bigoplus_{a\in B} \opid ~,~}
\end{equation}
where $B$ is a finite set.
The different copies of $\opid$ in the direct sum defining $A$ will be distinguished by adding an index, i.e.\ $\opid_b$, $b\in B$.
One then has the decomposition
\begin{equation}
T = \bigoplus_{a,b,c\in B} {}_a t_{bc} ~,
\end{equation}
where for $a,b,c\in B$, ${}_a t_{bc} \in \mcC$ is an object.
The bimodule structure of $T$ is such that ${}_a t_{bc}$ is an $\opid_a$-($\opid_b \otimes \opid_c$)-bimodule, 
where the action of the corresponding summand $\opid_d$, $d\in B$ is via the unitor morphisms in $\mcC$ and the other summands act by zero.
At this point we make one more
\begin{equation}
\tag{A3}
\label{eq:assumption3}
\boxed{~\text{assumption:} \raisebox{-.5em}{\rule{0pt}{1.7em}} \quad
\text{each ${}_a t_{bc}$ is either $0$ or a simple object.}~}
\end{equation}

\medskip

To obtain the rest of the orbifold datum, one must find the morphisms $\a, \abar, \psi$ and the scalar $\phi$ and verify conditions (O1)-(O8).
From the assumptions \eqref{eq:assumption1}--\eqref{eq:assumption3} one gets decompositions $\a = \bigoplus_{a,b,c,d\in B} \a^a_{bcd}$, $\abar = \bigoplus_{a,b,c,d\in B} \abar^a_{bcd}$, where
\begin{align}
\a^a_{bcd}&: \bigoplus_{p\in B} {}_a t_{bp} \otimes {}_p t_{cd} \longrightarrow \bigoplus_{q\in B} {}_a t_{qd} \otimes {}_q t_{bc}~, \quad
\\
\abar^a_{bcd}&: \bigoplus_{p\in B} {}_a t_{pd} \otimes {}_p t_{bc} \longrightarrow \bigoplus_{q\in B} {}_a t_{bq} \otimes {}_q t_{cd}~,
\label{eq:alpha_abcd}
\end{align}
as well as
\begin{equation}
\label{eq:psi_decomp}
\psi = \bigoplus_{a\in B} \psi_a \cdot \id_{\opid_a}~, \quad
\psi_a \in \opk^\times~.
\end{equation}
The conditions these data must satisfy are somewhat simpler if one uses the following rescaled version of morphisms $\a^a_{bcd}$ instead:
\begin{equation}
\label{eq:alpha_ito_f}
\a^a_{bcd}  = \frac{1}{\psi^2_c} \, f^a_{bcd}~, \quad a,b,c,d\in B~,
\end{equation}
where $f^a_{bcd}$ are morphisms having the same source and target as $\a^a_{bcd}$ in \eqref{eq:alpha_abcd}.
Similarly, we denote by $g^a_{bcd}$ a rescaling of $\abar^a_{bcd}$ (we will see later that this is the inverse of $f^a_{bcd}$):
\begin{equation}
\label{eq:abar_ito_g}
\abar^a_{bcd} \big|_{pq} = \frac{\psi^2_c}{\psi^2_p \psi^2_q} \, g^a_{bcd}\big|_{pq}~,
\end{equation}
where we use the notation
\begin{align}
& \a^a_{bcd} \big|_{pq}, ~f^a_{bcd} \big|_{pq} \,:\, {}_a t_{bp} \otimes {}_p t_{cd} \longrightarrow {}_a t_{qd} \otimes {}_q t_{bc}~, \nonumber \\
& \abar^a_{bcd} \big|_{pq}, ~g^a_{bcd} \big|_{pq} \,:\, {}_a t_{pd} \otimes {}_p t_{bc} \longrightarrow {}_a t_{bq} \otimes {}_q t_{cd}~,
\end{align}
for the restrictions of $\a^a_{bcd}$, $f^a_{bcd}$, $\abar^a_{bcd}$ and $g^a_{bcd}$ to the corresponding direct summands.
For $i\in I$, we introduce scalars $f^{a,~i}_{bcd, ~pq}$ and $g^{a, ~i}_{bcd, ~pq}$ such that
\begin{equation}
\label{eq:fg_scalars}
f^a_{bcd} \big|_{pq} = \sum_{i\in I} f^{a,~i}_{bcd, ~pq} \pic[2.0]{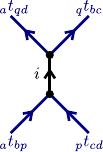}, \quad
g^a_{bcd} \big|_{pq} = \sum_{i\in I} g^{a,~i}_{bcd, ~pq} \pic[2.0]{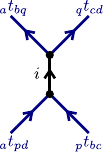}.
\end{equation}
One can now translate the conditions (O1)--(O8) into equations for these scalars. The result is:

\newcommand\nexteqn{\\[-.5em] \nonumber & \hspace{2em} \rule{28em}{0.2pt} \\ }
\begin{table}
\captionsetup{format=plain}
\centering
\begin{align}
\nonumber
& \sum_{i,j\in I}
F^{\left( {}_a t_{re} ~ {}_r t_{sd} ~ {}_s t_{bc} \right) g}_{k \, i} ~
G^{\left( {}_a t_{bp} ~ {}_p t_{cq} ~ {}_q t_{de} \right) g}_{j \, m} ~
F^{\left( {}_s t_{bc} ~ {}_a t_{sq} ~ {}_q t_{de} \right) g}_{i \, j} ~
R^{\left( {}_a t_{sq} ~ {}_s t_{bc} \right) j} ~
R^{-\left( {}_s t_{bc} ~ i \right) g} 
\\[-0.8em] \nonumber
& \hspace{35pt} \cdot
f^{a, ~i}_{sde, ~qr} ~
f^{a, ~j}_{bcq, ~ps}
\\ \tag{O1}\label{O1}
&= \sum_{x\in B, l\in I}
F^{\left( {}_a t_{re} ~ {}_r t_{bx} ~ {}_x t_{cd} \right) g}_{k \, l} ~
G^{\left( {}_a t_{bp} ~ {}_p t_{xe} ~ {}_x t_{cd} \right) g}_{l \, m} ~
f^{p, ~m}_{cde, ~qx} ~
f^{a, ~l}_{bxe, ~pr} ~
f^{r, ~k}_{bcd, ~xs}
\nexteqn
\tag{O2}\label{O2}
&\sum_{q\in B} f^{a,~i}_{bcd, ~qr} \cdot g^{a,~i}_{bcd, ~pq} = \d_{pr} ~ N_{{}_a t_{pd} ~ {}_p t_{bc}}^i
\nexteqn
\tag{O3}\label{O3}
&\sum_{q\in B} g^{a,~i}_{bcd,~qr}\cdot f^{a,~i}_{bcd,~pq} = \d_{pr} ~ N_{{}_a t_{bp} ~ {}_p t_{cd}}^i
\nexteqn
\nonumber
& \sum_{d\in B, i,j\in I}
\frac{\psi_b^2 \psi_d^2}{\psi_q^2 \psi_p^2} ~
f^{a, ~j}_{bcd, ~pq} ~
g^{a, ~i}_{b'cd, ~qp} ~
F^{\left( {}_p t_{cd} ~ m ~ {}_q t_{bc}  \right) j}_{ {}_a t_{bp} ~ {}_a t_{qd}} ~
G^{\left( {}_p t_{cd} ~ m ~ {}_q t_{b'c} \right) i}_{ {}_a t_{qd} ~ {}_a t_{b'p} }
\\[-.5em]
\tag{O4}\label{O4}
&\hspace{60pt}
\cdot
R^{\left( {}_p t_{cd} ~ {}_a t_{bp} \right) j} ~
R^{-\left( {}_a t_{b'p} ~ {}_p t_{cd} \right) i} ~
\frac{\dim j}{ \dim {}_a t_{qd}} \frac{\dim i}{\dim {}_a t_{bp}}
~=~ 
\d_{bb'} ~ N_{m ~ {}_q t_{bc}}^{{}_a t_{bp}}
\nexteqn
\nonumber
& \sum_{b\in B, i,j\in I}
\frac{\psi_b^2 \psi_d^2}{\psi_p^2 \psi_q^2} ~
g^{a, ~j}_{bcd, ~pq} ~
f^{a, ~i}_{bcd', ~qp} ~
F^{\left( {}_p t_{bc} ~ m ~ {}_q t_{cd} \right) j}_{{}_a t_{pd} ~ {}_a t_{bq}}~
G^{\left( {}_p t_{bc} ~ m ~ {}_q t_{cd'} \right) i}_{{}_a t_{bq} ~ {}_a t_{pd'}}
\\[-.5em]
\tag{O5}\label{O5}
&\hspace{60pt}
\cdot
R^{-\left( {}_p t_{bc} ~ {}_a t_{pd}  \right) j} ~
R^{\left( {}_a t_{pd'} ~  {}_p t_{bc} \right) i} ~
\frac{\dim j}{\dim {}_a t_{bq}}
\frac{\dim i}{\dim {}_a t_{pd}}
~=~
\d_{dd'} ~ N_{m ~ {}_q t_{cd}}^{{}_a t_{pd}}
\nexteqn
\nonumber
& \sum_{c\in B, i,j\in I}
\frac{\psi_a^2 \psi_c^2}{\psi_q^2 \psi_p^2} ~
f^{a, ~j}_{bcd, ~pq} ~
g^{a', ~i}_{bcd, ~qp} ~
G^{\left( {}_a t_{qd} ~ m ~ {}_p t_{cd} \right) j}_{{}_a t_{bp} ~ {}_q t_{bc}} ~
F^{\left( {}_{a'} t_{qd} ~ m ~ {}_p t_{cd} \right) i}_{{}_q t_{bc} ~ {}_{a'} t_{bp}}
\\[-.5em]
\tag{O6}\label{O6}
&\hspace{180pt}
\cdot\frac{\dim j}{\dim {}_q t_{bc}} \frac{\dim i}{\dim {}_a t_{bp}}
~=~
\d_{aa'} ~ N_{{}_a t_{qd} ~ m}^{{}_a t_{bp}}
\nexteqn
\nonumber
& \sum_{a\in B, i,j\in I}
\frac{\psi_a^2 \psi_c^2}{\psi_p^2 \psi_q^2} ~
g^{a, ~j}_{bcd, ~pq} ~
f^{a, ~i}_{bc'd, ~qp} ~
F^{\left( {}_a t_{pd} ~ m ~ {}_q t_{cd} \right) j}_{{}_p t_{bc} ~ {}_a t_{bq}} ~
G^{\left( {}_a t_{pd} ~ m ~ {}_q t_{c'd} \right) i}_{{}_a t_{bq} ~ {}_p t_{bc'}}
\\[-.5em]
\tag{O7}\label{O7}
&\hspace{180pt}
\cdot\frac{\dim j}{\dim {}_a t_{bq}}
\frac{\dim i}{\dim {}_p t_{bc}}
~=~
\d_{cc'} ~ N_{m ~ {}_q t_{cd}}^{{}_p t_{bc}}
\nexteqn
\tag{O8}\label{O8}
&\sum_{b,c\in B} \psi_b^2 \psi_c^2 \dim {}_a t_{bc} =
\sum_{b,c\in B} \psi_b^2 \psi_c^2 \dim {}_b t_{ca} =
\sum_{b,c\in B} \psi_b^2 \psi_c^2 \dim {}_c t_{ab} =
\psi_a^2 \phi^{-2}
\end{align}
\caption{Polynomial equations defining an orbifold datum $\opA$ in a modular fusion category $\mcC$ under assumptions \eqref{eq:assumption1}--\eqref{eq:assumption3}.
The sum over $d\in B$ in \eqref{O4} is restricted to those $d$ for which ${}_at_{qd} \neq 0$. Analogous restrictions apply to the sums over $B$ in \eqref{O5}--\eqref{O7}.}
\label{tab:orb_eqs}
\end{table}

\begin{prp}\label{prop:orb-polys}
Under the assumptions \eqref{eq:assumption1}, \eqref{eq:assumption2}, \eqref{eq:assumption3}, giving an orbifold datum in $\mcC$ is equivalent to giving a set of scalars
\begin{equation}
f^{a,~i}_{bcd, ~pq}, \quad g^{a, ~i}_{bcd, ~pq}, \quad \psi_a, \quad \phi, \quad a,b,c,d,p,q\in B, \quad i\in I,
\end{equation}
which satisfy the equations in Table \ref{tab:orb_eqs}.
\end{prp}

\begin{proof}
The reformulation of conditions (O1)--(O8) from \cite[Def.\,2.2]{Mulevicius:2020bat} in terms of the scalars defining the orbifold datum under assumptions \eqref{eq:assumption1}--\eqref{eq:assumption3} is tedious but straightforward.
The labelling conventions we use are shown in Appendix~\ref{app:convention}.
As an example, the computation for the pentagon identity (O1) is given in Appendix~\ref{app:pentagon}.
The derivation of the other identities is similar and we omit the details.
\end{proof}

\begin{rem}
The identities \eqref{O2} and \eqref{O3} show that 
\begin{equation}
f^a_{bcd}~:~ \bigoplus_{p\in B} {}_a t_{bp} \otimes {}_p t_{cd} ~\longleftrightarrow~ \bigoplus_{q\in B} {}_a t_{qd} \otimes {}_q t_{bc} ~:~ g^a_{bcd}
\end{equation}
are indeed inverse to each other. In particular, the scalars $g^{a, ~i}_{bcd, ~pq}$ are uniquely determined by the scalars $f^{a,~i}_{bcd, ~pq}$.
\end{rem}

\medskip

Let us recall a (rather strong) notion of an isomorphism of orbifold data in $\mcC$ as introduced in \cite[Def.\,3.12]{CRS3}:
A \textit{$T$-compatible isomorphism} from an orbifold datum $\opA = (A, T, \a, \abar, \psi, \phi)$ to $\widetilde{\opA} = (A, \widetilde{T}, \widetilde{\a}, \widetilde{\abar}, \psi, \phi)$ is an isomorphism $\rho: T \ra \widetilde{T}$ of $A$-$A\otimes A$-bimodules such that
\begin{equation}
\label{eq:T-compatible_iso}
(\rho \otimes \rho)\circ\a = \widetilde{\a} \circ (\rho \otimes \rho)~.
\end{equation}
The orbifold TQFTs obtained from two orbifold data related by a $T$-compatible isomorphism are isomorphic, see \cite[Lem.\,3.13]{CRS3}.

Given a scalar $\xi\in\opk^\times$ one can define a new orbifold datum
\begin{equation}
\label{eq:rescaled-orbifold}
    \opA_\xi = \big(T,\, A, \,\xi\,\a,\, \xi\,\abar,\, \xi^{-1/2}\psi,\, \xi^{1/2}\phi\big) ~.
\end{equation}
It is easy to check that $\opA_\xi$ again satisfies the conditions (O1)-(O8) and that the corresponding orbifold TQFTs are isomorphic.
We will refer to $\opA_\xi$ as a \textit{rescaling} of $\opA$.

\medskip

Consider a $T$-compatible isomorphism $\rho: T \ra T$ of orbifold data satisfying the assumptions 
\eqref{eq:assumption1}--\eqref{eq:assumption3}.
In this case one has:
\begin{equation}\label{eq:rho-direct-sum}
\rho = \bigoplus_{a,b,c\in B} {}_a \l_{bc} \cdot \id_{{}_a t_{bc}}~,
\end{equation}
where there is one scalar ${}_a \l_{bc}\in\opk^\times$ for each non-zero ${}_a t_{bc}$.
It follows from \eqref{eq:T-compatible_iso} that $\a$ and $\widetilde{\a}$ are related by
\begin{equation}
\widetilde{\a}^a_{bcd} \big|_{pq} = 
\frac{{}_a \l_{qd} ~ {}_q \l_{bc}}{{}_a \l_{bp} ~ {}_p \l_{cd}} \cdot
\a^a_{bcd} \big|_{pq}, \quad
a,b,c,d,p,q\in B ~,
\end{equation}
or equivalently, the sets of scalars $f^{a, ~i}_{bcd, ~pq}$, $g^{a, ~i}_{bcd, ~pq}$, $\widetilde{f}^{a, ~i}_{bcd, ~pq}$, $\widetilde{g}^{a, ~i}_{bcd, ~pq}$ are related by
\begin{equation}
\label{eq:f_ftilde_rels}
\widetilde{f}^{a, ~i}_{bcd, ~pq} = 
\frac{{}_a \l_{qd} ~ {}_q \l_{bc}}{{}_a \l_{bp} ~ {}_p \l_{cd}} \cdot
f^{a, ~i}_{bcd, ~pq} ~ \quad
\widetilde{g}^{a, ~i}_{bcd, ~pq} =
\frac{{}_a \l_{bq} ~ {}_q \l_{cd}}{{}_a \l_{pd} ~ {}_p \l_{bc}} \cdot
g^{a, ~i}_{bcd, ~pq} ~.
\end{equation}
Evidently, if $f^{a, ~i}_{bcd, ~pq}$, $g^{a, ~i}_{bcd, ~pq}$, $\psi_a$, $\phi$ solve the equations in Table \ref{tab:orb_eqs}, then so do $\widetilde{f}^{a, ~i}_{bcd, ~pq}$, $\widetilde{g}^{a, ~i}_{bcd, ~pq}$, $\psi_a$, $\phi$.

\medskip

One can exploit the invariance of equations in Table \ref{tab:orb_eqs} under transformations \eqref{eq:f_ftilde_rels} to simplify the search for solutions. Will will illustrate this in the next lemma for solutions which in addition satisfy the following unitality condition:
\begin{equation}
\tag{A4}
\label{eq:assumption4}
 \quad
 \boxed{
\begin{array}{ll}
\text{assumption:} &
\text{there is a distinguished element $\varone \in B$, such that}\\
&
{}_a t_{\varone b} = {}_a t_{b\varone} = 
\begin{cases}
\opid ~\text{, if $a=b$}\\
0 ~\text{, if $a\neq b$}
\end{cases}.
\end{array}
}
\end{equation}

\begin{lem}
\label{lem:gauge_transfs}
Suppose assumptions \eqref{eq:assumption1}--\eqref{eq:assumption4} hold. 
Let $a,b,c\in B$ be such that ${}_at_{bc} \neq 0$.
Then the morphisms $f^a_{\varone bc}$, $f^a_{b\varone c}$, $f^a_{bc\varone}$ are determined by the scalars
$f^{a, ~ {}_a t_{bc}}_{\varone bc, ~ ab}$,
$f^{a, ~ {}_a t_{bc}}_{b\varone c, ~ cb}$,
$f^{a, ~ {}_a t_{bc}}_{bc\varone , ~ ca}$, all of which are non-zero. Via a suitable transformation of the form \eqref{eq:f_ftilde_rels}, we can achieve that
\begin{equation}\label{eq:1-normalisation-condition}
    f^{a, ~ {}_a t_{bc}}_{\varone \, b\,c, ~ ab} ~=~ 
f^{a, ~ {}_a t_{bc}}_{b\,\varone\,c, ~ cb} ~=~
f^{a, ~ {}_a t_{bc}}_{b\,c\,\varone , ~ ca} ~=~ 1 ~.
\end{equation}
\end{lem}

\begin{proof}
One quickly determines that the expression in \eqref{eq:fg_scalars} e.g.\ for $f^a_{bc\varone }$ has a single term when $p=c$, $q=a$ and vanishes otherwise.
That the corresponding scalars are non zero is implied by the invertibility conditions \eqref{O2} and \eqref{O3}.

The normalisation can be shown as follows:
Take $c=d=\varone $, $p=q=e$, $r=s=b$, $m=k=\opid$, $g = {}_a t_{be}$ in the condition \eqref{O1}.
Then assuming ${}_a t_{be} \neq 0$ one can simplify the resulting equation to
\begin{equation}
f^{a, ~ {}_a t_{be}}_{b\varone e, ~eb} = f^{e, ~ \opid}_{\varone \varone e, ~ e\varone } \cdot f^{b, ~ \opid}_{b\varone \varone , ~ \varone b}.
\end{equation}
Any transformation \eqref{eq:f_ftilde_rels} such that for all $b,e\in B$ one has
\begin{equation}
{}_b \l_{b\varone } = \left( f^{b, ~ \opid}_{b\varone \varone , ~\varone b} \right)^{-1}, \quad
{}_e \l_{\varone e} = f^{e, ~\opid}_{\varone \varone e, ~ e\varone }
\end{equation}
then results in $\widetilde{f}^{a, ~ {}_a t_{be}}_{b\varone e, ~eb} = 1$.
After this transformation, setting $b=c=s=\varone$, $p=q=a$, $r=d$, $k=\opid$, $m = g = {}_a t_{de}$, in \eqref{O1} and assuming ${}_a t_{de} \neq 0$ one finds that $\widetilde{f}^{a, ~ {}_a t_{de}}_{\varone de, ~ ad} = 1$ already holds.
Similarly, by taking $d = e = q = \varone$, $r=s=a$, $p=c$, $m=\opid$, $k = g = {}_a t_{bc}$ and assuming ${}_a t_{bc} \neq 0$ one finds that $\widetilde{f}^{a, ~ {}_a t_{bc}}_{bc\varone , ~ ca} = 1$ holds too.
\end{proof}

\section{The category $\mcC_\opA$ and some of its basic properties}\label{sec:CA-properties}

In this section we recall the condition on an orbifold datum $\opA$ to be simple, and how a simple orbifold datum $\opA \in \mcC$ gives rise to a new modular category $\mcC_\opA$. Under the simplifying assumptions \eqref{eq:assumption1}--\eqref{eq:assumption3} we give expressions in terms of the scalars in Proposition~\ref{prop:orb-polys} that determine whether $\opA$ is simple, compute the global dimension of $\mcC_\opA$, and give the number of simple objects of $\mcC_\opA$.

\medskip

As described in \cite{Mulevicius:2020bat}, an orbifold datum $\opA$ in a modular fusion category $\mcC$ yields a new (multi-)fusion category $\mcC_\opA$.
Its objects are $A$-$A$-bimodules $M$ equipped with $A$-$A\otimes A$-bimodule isomorphisms
\begin{equation}
\tau_1^M: M \otimes_A T \ra T \otimes_1 M, \quad
\tau_2^M: M \otimes_A T \ra T \otimes_2 M,
\end{equation}
which are called \textit{$T$-crossings} and satisfy a list of conditions available in \cite[Sec.\,3.1]{Mulevicius:2020bat}.
Morphisms in $\mcC_\opA$ are morphisms of the underlying $A$-$A$-bimodules which commute with the $T$-crossings.
The tensor product in $\mcC_\opA$ is the tensor product over $A$ of underlying bimodules with $T$-crossings as described in \cite[Sec.\,3.2]{Mulevicius:2020bat}
The tensor unit is given by the regular bimodule $A$ together with the $T$-crossings
\begin{equation}
\label{eq:T-cross_for_A}
\tau_i^A = \pic[2.0]{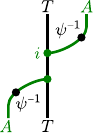}, \quad
\overline{\tau}_i^A = \pic[2.0]{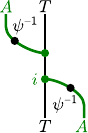} ~, \quad i=1,2 ~,
\end{equation}
where the index shows which right $A$-action on $T$ we use and $\overline{\tau}_i^A$ 
is the pseudo-inverse of $\tau_i$.
For an object $M \in \mcC_\opA$, the pseudo-inverse $\overline{\tau}_i^M$ of the $T$-crossing $\tau_i^M$  is inverse to $\tau_i^M$ up to an action of $\psi$'s, see \cite[Sec.\,3.1]{Mulevicius:2020bat}.

\medskip

We call the orbifold datum $\opA$ \textit{simple} if the tensor unit of $\mcC_\opA$ is simple, that is, if
$\dim \mcC_\opA(A,A) = 1$.
We have \cite[Thm.\,3.17]{Mulevicius:2020bat}:
\begin{thm}\label{thm:CA-is-modular}
Let $\opA$ be a simple orbifold datum in a modular fusion category $\mcC$.
Then $\mcC_{\opA}$ is again a modular fusion category.
\end{thm}
Recall from \eqref{eq:T-compatible_iso} and \eqref{eq:rescaled-orbifold} the notion of a $T$-compatible isomorphism and of a rescaled orbifold datum.

\begin{prp}\label{prop:T-com-rescale-equivcat}
Let $\opA$, $\widetilde{\opA}$ and $\opA_\xi$ be orbifold data in $\mcC$, such that $\widetilde{\opA}$ and $\opA$ are related by a $T$-compatible isomorphism, and such that $\opA_\xi$ is a rescaling of $\opA$ for some $\xi \in \opk^\times$.
Then $\mcC_\opA$, $\mcC_{\widetilde{\opA}}$ and $\mcC_{\opA_\xi}$ are equivalent as ribbon fusion categories.
\end{prp}
\begin{proof}
Given a $T$-compatible isomorphism $\rho:T \ra \widetilde{T}$,
define the functor $F:\mcC_\opA \ra \mcC_{\widetilde{\opA}}$ as $(M, \tau_1^M, \tau_2^M) \mapsto (M, \widetilde{\tau}_1^M, \widetilde{\tau}_1^M)$, where
\begin{equation}
\widetilde{\tau}_i^M := [M \otimes_A \widetilde{T} \xra{\id_M\otimes \rho^{-1}} M \otimes_A T \xra{\tau_i} T \otimes_i M \xra{\rho \otimes \id_M} \widetilde{T} \otimes_i M], \quad
i=1,2 ~.
\end{equation}
On morphisms, $F$ acts as the identity.
One checks that $\widetilde{\tau}_i$, $i=1,2$ are indeed $\widetilde{T}$-crossings.
For $M,N\in\mcC_\opA$ the objects $F(M \otimes N)$ and $F(M)\otimes F(N)$ are equal, giving $F$ a natural monoidal structure, which preserves braidings and twists.

Similarly the functor $\mcC_\opA \ra \mcC_{\opA_\xi}$, $(M, \tau_1^M, \tau_2^M) \mapsto (M, \xi\tau_1^M, \xi\tau_2^M)$ gives a ribbon equivalence between $\mcC_\opA$ and $\mcC_{\opA_\xi}$.
\end{proof}

Given two objects $M,N\in\mcC_\opA$ and a morphism $f:M\ra N$ of the underlying bimodules, let us define the \textit{average} of $f$ to be
\begin{equation}
\overline{f} := \phi^4 \pic[2.0]{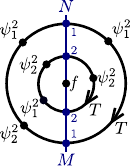} ~.
\end{equation}
Here, $\psi_1,\psi_2:T \ra T$ are morphisms obtained by composing $[\opid_\mcC \xra{\eta} A \xra{\psi} A]$ with the appropriate right action of $A$ on $T$ (where $\eta:\opid_\mcC\ra A$ is the unit of $A$; see \cite[Def.\,2.2]{Mulevicius:2020bat} for details on this notation).
The index at a crossing of a bimodule and a $T$-line refers to $\tau_i$ or $\overline \tau_i$, depending on the direction of the crossing.

\medskip

The average $\overline{f}$ of $f$ is a morphism in $\mcC_\opA$, see \cite[Sec.\,3.4]{Mulevicius:2020bat}.
In fact, the map $f \mapsto \overline{f}$ is an idempotent projecting onto the subspace $\mcC_\opA(M,N) \subseteq \mcACA(M,N)$.
Thus, $\opA$ is simple if and only if the image of the averaging map is one-dimensional when applied to bimodule morphisms $A\ra A$.

\begin{prp}\label{prop:A-simple}
Suppose that $\mathrm{char}(\opk) = 0$.
Under the assumptions \eqref{eq:assumption1}--\eqref{eq:assumption3} we have
\begin{equation}
\label{eq:dimCAAA}
\dim \mcC_\opA(A,A)
~=~
\phi^4 
\hspace{-.5em}
\sum_{a,b,d,p\in B} 
\hspace{-.5em}
\psi_b^2  \, \psi_d^2 \, \dim {}_a t_{pd} \, \dim {}_p t_{ba} ~.
\end{equation}
In particular, $\opA$ is simple if and only if the right hand side yields 1.
\end{prp}
\begin{proof}
Using the $T$-crossings for $A$ as in \eqref{eq:T-cross_for_A} one has:
\begin{equation}
\label{eq:average_comp}
\pic[2.0]{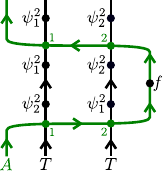} =
\pic[2.0]{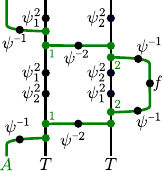} = 
\pic[2.0]{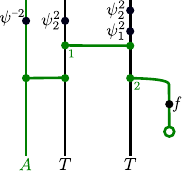}.
\end{equation}
The assumptions \eqref{eq:assumption1}--\eqref{eq:assumption3} allow us to write a morphism $f:A\ra A$ as $f = \bigoplus_{a\in B} f_a \id_{\opid_a}$ for some scalars $f_a$.
Substituting also $A = \bigoplus_{a\in B} \opid_a$ and $T = \bigoplus_{a,b,c\in B} {}_a t_{bc}$ ~, the right hand side of \eqref{eq:average_comp} becomes
\begin{equation}
\bigoplus_{a,b,c,d,p\in B} \frac{\psi_b^2 \psi_c^2 \psi_d^2}{\psi_a^2} f_c \cdot \id_{\opid_a} \otimes \id_{{}_a t_{pd}} \otimes \id_{{}_p t_{bc}}
~.
\end{equation}
Multiplying with $\phi^4$ and taking the trace over $T\otimes T$ yields
\begin{equation}
\overline{f} =
\bigoplus_{a\in B} \left(
\sum_{c\in B} f_c \cdot \left[ 
\phi^4
\sum_{b,d,p\in B} \frac{\psi_b^2 \psi_c^2 \psi_d^2}{\psi_a^2} \cdot 
\dim {}_a t_{pd} \cdot \dim {}_p t_{bc}
\right]
\right)
\id_{\opid_a} .
\end{equation}
In the basis $\{ \id_{\opid_a} \}_{a\in B}$ of $\mcACA(A,A)$, the numbers in the square brackets describe precisely the matrix elements of the averaging map, indexed by $a,c \in B$.
Taking the trace of this matrix then yields exactly the right hand side of \eqref{eq:dimCAAA}.
\end{proof}

Next we turn to the global dimension of $\mcC_\opA$.
The global dimension of a spherical fusion category is the sum of the squared dimensions of its simple objects.
Given a simple orbifold datum $\opA$ in a modular fusion category $\mcC$, it was shown in \cite[Thm.\,3.17]{Mulevicius:2020bat} that the global dimension of $\mcC_\opA$ is
\begin{equation}
    \Dim\mcC_\opA \,=\, \frac{\Dim\mcC}{\phi^8 \cdot (\tr_\mcC \psi^4)^2}
    ~\overset{(*)}=~
    \frac{\sum_{i\in I} (\dim i)^2}{
    \phi^8 \cdot \big( \sum_{b \in B} \psi_b^4\big)^2}
    ~,
\end{equation}
where in $(*)$ we used assumptions \eqref{eq:assumption1}--\eqref{eq:assumption3} and substituted the scalars in Proposition~\ref{prop:orb-polys}.

\medskip

Finally, we turn to computing the number of (isomorphism classes of) simple objects in $\mcC_\opA$. 
We will do this with the help of
Reshetikhin-Turaev TQFT:

\begin{lem}\label{lem:T3-gives-number}
Suppose that $\mathrm{char}(\opk) = 0$.
The invariant of the three-torus $T^3=S^1 \times S^1 \times S^1$ in the Reshetikhin-Turaev TQFT $Z_\mcC$ for a modular fusion category $\mcC$ is equal to the number of isomorphism classes of simple objects in $\mcC$:
\begin{equation}
    Z_\mcC(T^3) ~=~ |I| \ .
\end{equation}
\end{lem}
\begin{proof}
This is immediate from writing $Z_\mcC(T^3)$ as the trace over $Z_\mcC(T^2 \times [0,1])$. The trace computes the dimension of $Z_\mcC(T^2) = \mcC(\opid,\bigoplus_{i \in I} i^* \otimes i)$, which in turn gives the number of simple objects, see \cite[Sec.\,IV.1.4]{Tu}.
\end{proof}

Next we use the equivalence of the TQFT for $\mcC_\opA$ and the orbifold TQFT of $\mcC$ by $\opA$ \cite{Carqueville:2020}:
\begin{thm}
For a modular fusion category $\mcC$ and a simple orbifold datum $\opA$ in $\mcC$, the TQFTs $Z_{\mcC_\opA}$ and $Z^\mathrm{orb,\opA}_\mcC$ are equivalent. In particular, for each closed oriented three-manifold $M$, we have
\begin{equation}
    Z_{\mcC_\opA}(M) ~=~ Z^\mathrm{orb,\opA}_\mcC(M) \ .
\end{equation}
\end{thm}

We will not review the construction of the orbifold TQFT $Z^\mathrm{orb,\opA}_\mcC$ and instead refer to \cite{CRS3} for details. The following lemma relies heavily on the methods in \cite{CRS2,CRS3} and gives a complicated but explicit expression for the invariant assigned to $T^3$ by the orbifold TQFT.

\begin{lem}
\label{lem:T3_invariant}
For a modular fusion category $\mcC$ and a simple orbifold datum $\opA$ satisfying assumptions \eqref{eq:assumption1}--\eqref{eq:assumption3} from Section~\ref{sec:poly}, 
we have
\begin{align} \nonumber
Z^\mathrm{orb,\opA}_\mcC & (T^3) =\\ \nonumber
&\phi^2 \cdot \sum_{\substack{a,b,c,d,e,f,g\in B \\ r,s,t,u,v,w\in I}}
\frac{\psi_a^2 \psi_c^2 \psi_e^2 \psi_f^2}{\psi_b^2 \psi_d^2 \psi_g^2} ~
f^{e, ~r}_{afc, ~bg} ~ f^{e, ~s}_{caf, ~gd} ~ f^{e, ~u}_{fca, ~db} ~
g^{e, ~t}_{acf, ~db} ~ g^{e, ~v}_{fac, ~gd} ~ g^{e, ~w}_{cfa, ~bg}\\ \nonumber
& \hspace{10pt} \cdot
\sum_{x,y,z,k,l,m\in I}
L^{bcf, ~wt}_{eba, ~ur} \big|_x ~
L^{dca, ~us}_{efd, ~vt} \big|_y ~
L^{gaf, ~sr}_{ecg, ~wv} \big|_k \\ \nonumber
& \hspace{65pt} \cdot
\frac{\dim z}{\dim l} ~\frac{\dim v}{\dim k} ~
F^{(r \, z \, y)v}_{k \, u} ~
F^{(s \, k \, x)u}_{l \, w} ~
G^{(t \, x \, z)u}_{r \, m} ~
G^{(s \, y \, m)u}_{t \, l}\\ \label{eq:orb-simple}
& \hspace{65pt} \cdot
T_{x y z, ~ k l m} ~,
\end{align}
where
\begin{align*}
& L^{bcf, ~wt}_{eda, ~ur}\big|_x =
N^{w}_{{}_e t_{da} ~ {}_b t_{cf}} ~
N^{t}_{{}_e t_{ad} ~ {}_b t_{cf}} ~
\frac{\dim x}{\dim {}_b t_{fc}} ~
F^{({}_e t_{ad} ~ {}_b t_{cf} ~ x)r}_{{}_b t_{fc} ~ t} ~
G^{({}_e t_{da} ~ {}_b t_{cf} ~ x)u}_{w ~ {}_b t_{fc}} ~,\\
& T_{xyz, ~ klm} =
\sum_{p,j\in I}
\frac{\dim j}{\dim p} ~
\frac{\theta_j}{ \theta_p \theta_x} ~
R^{-(z \, x)m} ~
G^{(p \, z \, y)p}_{p \, k} ~
G^{(p \, k \, x)j}_{p \, l} ~
F^{(p \, y \, m)j}_{l \, p} ~
F^{(p \, z \, x)j}_{m \, p} ,
\end{align*}
and the expression for the $L$-symbol holds when ${}_b t_{fc} \neq 0$ and is set to zero otherwise.
\end{lem}
The proof is a rather technical computation which we present in Appendix \ref{app:T3_invariant}
Combining the above results, we arrive at the following statement.

\begin{prp}\label{prop:count-simple}
Suppose that $\mathrm{char}(\opk) = 0$.
For modular fusion category $\mcC$ and a simple orbifold datum $\opA$ satisfying assumptions \eqref{eq:assumption1}--\eqref{eq:assumption3} from Section~\ref{sec:poly}, the number of isomorphism classes of simple objects in $\mcC_\opA$ is given by \eqref{eq:orb-simple}.
\end{prp}

\begin{rem}
In Proposition~\ref{prop:A-simple}, Lemma~\ref{lem:T3-gives-number} and Proposition~\ref{prop:count-simple} we had to require $\mathrm{char}(\opk)=0$ because we computed the dimension of a vector space as a trace. If $\opk$ has finite characteristic, the corresponding equalities still hold modulo $\mathrm{char}(\opk)$. That is, if $\opA$ is simple, then the right hand side of \eqref{eq:dimCAAA} is 1 modulo $\mathrm{char}(\opk)$ (but the converse no longer holds), and the number of isomorphism classes of simple objects in $\mcC_\opA$ is given by \eqref{eq:orb-simple} modulo $\mathrm{char}(\opk)$.
\end{rem}

\section{Ising-type modular categories}\label{sec:Ising}
For the rest of the paper, we work over the field of complex numbers, $\opk=\opC$.

A braided fusion category of Ising-type is a special case of a Tambara-Yamagami category, where the underlying abelian group is $\mathbb{Z}_2$ \cite{Tambara:1998,Siehler:2001}. Including also the ribbon twist, there are 16 different modular fusion categories of Ising-type \cite[App.\,B]{DGNO}, 
\begin{equation}
\label{eq:I_zeta_epsilon}
    \mcI_{\zeta,\epsilon}
    \quad ,  \qquad \text{where}
    \quad \zeta^8=-1 ~,~~ \epsilon \in \{\pm1 \} \ .
\end{equation}
We will use the abbreviation
\begin{equation}
\label{eq:lambda_ito_zeta}
    \lambda = \zeta^2 + \zeta^{-2} ~,
\end{equation}
which implies that $\lambda^2=2$. 
The category $\mcI_{\zeta,\epsilon}$
has three simple objects, $I=\{\opid, \vareps, \sigma \}$. 
Here, $\opid,\vareps$ form the $\mathbb{Z}_2$-subgroup of the fusion ring, and $\sigma \otimes \sigma \cong \opid \oplus \vareps$.
The quantum dimensions and twist eigenvalues of the simple objects are
\begin{equation}
    \begin{array}{rclrclrcl}
    \dim(\opid)&=&1 \ ,~~ &
    \dim(\vareps)&=&1 \ ,~~ &
    \dim(\sigma)&=&\epsilon\lambda \ , 
    \\
    \theta_\opid &=& 1 \ , &
    \theta_\vareps &=& -1 \ , &
    \theta_\sigma &=& \epsilon \zeta^{-1} \ .
    \end{array}
\end{equation}

The $R$-matrices of $\mcI_{\zeta,\epsilon}$ are given by
\begin{equation}
R^{(\vareps\vareps)\opid} = -1~,~~
  R^{(\sigma\sigma)\opid} =\zeta ~,~~
  R^{(\sigma\sigma)\vareps} =\zeta^{-3} ~,~~
  R^{(\sigma\vareps)\sigma} =   R^{(\vareps\sigma)\sigma} =\zeta^4 \ .
\end{equation}
The $F$-matrices with one internal channel are
\begin{equation}
\begin{array}{l}
    F^{(\vareps\vareps\vareps)\vareps}_{\opid\opid} = 1
    ~,~~
    \\
    F^{(\vareps\vareps\sigma)\sigma}_{\sigma\opid} = 1
    ~,~~
    F^{(\vareps\sigma\vareps)\sigma}_{\sigma\sigma} = -1
    ~,~~
    F^{(\sigma\vareps\vareps)\sigma}_{\opid\sigma} = 1 \ ,
    \\
    F^{(\sigma\sigma\vareps)\vareps}_{\sigma\opid} = 1
    ~,~~
    F^{(\sigma\vareps\sigma)\vareps}_{\sigma\sigma} = -1
    ~,~~
    F^{(\vareps\sigma\sigma)\vareps}_{\opid\sigma} = 1 \ .
\end{array}
\end{equation}
as well as $F^{(ijk)\opid}_{ik}=1$ for $i,j,k \in I$ whenever the $F$-matrix is allowed by fusion, i.e.\ when $\opid$ is a summand of $i \otimes j \otimes k$. Finally, the only $F$-matrix with two internal channels is
\begin{equation}
    F^{(\sigma\sigma\sigma)\sigma}_{\opid\opid}
    =
    F^{(\sigma\sigma\sigma)\sigma}_{\opid\vareps}
    =
    F^{(\sigma\sigma\sigma)\sigma}_{\vareps\opid}
    = 
     \frac{1}{\lambda}
    ~,~~
    F^{(\sigma\sigma\sigma)\sigma}_{\vareps\vareps} =  -\frac{1}{\lambda} 
    \ .
\end{equation}
The $G$-matrix in this case is obtained from the relation $G^{(ijk)l}_{pq} = F^{(kji)l}_{pq}$ (see e.g.~\cite[Eqn.\,(2.61)]{FRS1}).

The global dimension of $\mcI_{\zeta,\epsilon}$ and its anomaly (or Gauss sum, 
see e.g.\ \cite[Sec.\,6.2]{DGNO}) are given by 
\begin{align}
\mathrm{Dim}(\mcI_{\zeta,\epsilon}) &= \sum_{i \in I} \dim(i)^2 = 4
~~,
\nonumber\\
\xi(\mcI_{\zeta,\epsilon}) &= \frac{1}{\sqrt{\mathrm{Dim}(\mcI_{\zeta,\epsilon})}} \sum_{i \in I} \dim(i)^2 \,\theta_i = \epsilon \zeta^{-1} = \exp\!\big(2 \pi i \tfrac{c}{8}\big) \ .
\label{eq:Ising-Dim-anomaly}
\end{align}
The number $c$ is also called topological central charge and is defined mod 8. For $\zeta = \exp(-\pi i \frac18 - 2 \pi i \frac{m}8)$, $m=0,1,\dots,7$ one finds $c = 4 \delta_{\epsilon,-1} + \frac12 + m \ (\text{mod}\,8)$.

\begin{rem}\label{rem:which-zeta}~
\begin{enumerate}
\item
The ``original'' Ising modular category, i.e.\ the representation category of the $c=\frac12$ unitary Virasoro vertex operator algebra, is given by $\epsilon=1$ and $\zeta=\exp(-\pi i \frac18)$  as only for these values we have 
\begin{equation}
    \dim(\sigma)=\sqrt{2}
    \quad  , \quad
    \theta_\sigma = \exp\!\big(2 \pi i \tfrac1{16}\big)
    \ .
\end{equation}
\item 
For $\mcC(sp(4),1)$ which appears in the extension of $\mcC(sl(2),10)$ as in Remark~\ref{rem:E6}
we have \cite[Eq.\,(17.80)]{DiFr}
\begin{equation}
    \dim(\sigma)=\sqrt{2}
    \quad  , \quad
    \theta_\sigma = \exp\!\big(2 \pi i \tfrac5{16}\big) \ ,
\end{equation}
which corresponds to $\epsilon=-1$ and $\zeta=\exp(\pi i \frac38)$. The central charge of the chiral $Sp(4)$ WZW model at level 1 is $c=\frac52$, which agrees with the topological central charge (mod 8). 
\end{enumerate}
\end{rem}

\section{Fibonacci-type solutions inside Ising categories}\label{sec:orbdata}

Here we find all solutions for orbifold data in Ising-type categories for a particular ansatz for $A$ and $T$. 

\medskip

Fix $\zeta$ and $\epsilon$ as in Section~\ref{sec:Ising}. We will work in the modular fusion category $\mcI_{\zeta,\epsilon}$. 
We make the ansatz $B = \{ \varone , \varphi \}$ and
\begin{equation}\label{eq:Fib-sol-1}
    A = \opid_\varone \oplus \opid_\varphi
    ~~,\quad
	{}_at_{bc} = \begin{cases} \opid &;~ \text{either 0 or 2 of $a,b,c$ are $\varphi$}
	\\
	\sigma &;~ \text{all of $a,b,c$ are $\varphi$}
	\\
0 &;~ \text{else}
	\end{cases}
\end{equation}
As mentioned below \eqref{eq:t-ansatz}, this mimics the fusion rules of a Fibonacci category, hence we call the solutions below of Fibonacci type. 

Let $h \in \opC$ satisfy
\begin{equation}\label{eq:Fib-sol-2}
    h^3 = \zeta
    \qquad \text{and} \qquad
    h \text{ is a primitive $48^\text{th}$ root of unity} \ .
\end{equation}
We fix the following values for $f$, $\psi$, $\phi^2$:
\begin{align}
&     \psi^2_\varone \phi^2 = \frac{1}{3 - h^4 - h^{-4}}
     ~~,\quad
    \psi^2_\varphi \phi^2 = -\frac{h^{10} + h^{-10}}{3-h^4 - h^{-4}} \cdot \epsilon \ ,
 \nonumber    \\ &     
 f_{\varphi\varphi\varphi,~\varphi\varphi}^{\varone,~\sigma} = h
     ~~,\quad
     f_{\varphi\varphi\varphi,~\varphi\varphi}^{\varphi,~\vareps} = h^5
     ~~,\quad
     f_{\varphi\varphi\varphi,~\varone\varone}^{\varphi,~\opid} = \frac{1}{h^{12} (h^2 - h^{-2})}
     ~~,\quad
 \nonumber    \\ &     
     f_{\varphi\varphi\varphi,~\varphi\varphi}^{\varphi,~\opid} = -h^{-1}\,f_{\varphi\varphi\varphi,~\varone\varone}^{\varphi,~\opid}
     ~~,\quad
     f_{\varphi\varphi\varphi,~\varone\varphi}^{\varphi,~\opid} \, f_{\varphi\varphi\varphi,~\varphi \varone}^{\varphi,~\opid} = \frac{\lambda}{h} \,f_{\varphi\varphi\varphi,~\varone\varone}^{\varphi,~\opid} \ .
\label{eq:Fib-sol-3}
\end{align}
The value of $\phi^2 \in \opC^\times$ can be chosen arbitrarily and then fixes those of $\psi_\varone^2$, $\psi_\varphi^2$.
Similarly, the value of, say,  $f_{\varphi\varphi\varphi,~\varone\varphi}^{\varphi,~\opid} \in \opC^\times$ is arbitrary, fixing that of $f_{\varphi\varphi\varphi,~\varphi\varone}^{\varphi,~\opid}$.

\begin{thm}\label{thm:main-detail}
Every orbifold datum in $\mcI_{\zeta,\epsilon}$ which 
has $A,T$ as specified in \eqref{eq:Fib-sol-1} and 
satisfies the normalisation condition \eqref{eq:1-normalisation-condition} is given by \eqref{eq:Fib-sol-3} with $h$ subject to \eqref{eq:Fib-sol-2}.
Different choices for $\phi^2, f_{\varphi\varphi\varphi,~\varone\varphi}^{\varphi,~\opid} \in \opC^\times$ 
are related by rescalings and by $T$-compatible isomorphisms.
\end{thm}
\begin{proof}
Let us first note that by Lemma \ref{lem:gauge_transfs} one can set
\begin{equation}
\label{eq:fs_init1}
f^{a, ~i}_{bcd, ~pq} = N_{{}_a t_{bp} ~ {}_p t_{cd}}^i ~ N_{{}_a t_{qd} ~ {}_q t_{bc}}^i~, \quad 
a,b,c,d,p,q\in B~, \quad i\in I
\end{equation}
whenever at least one of $b,c,d$ is $\varone\in B$.
For the rest of the $f$-coefficients which are not automatically zero by the fusion rules we use the abbreviations
\begin{gather}
    \nonumber
f^{\varone, ~\sigma}_{\varphi\varphi\varphi, ~\varphi\varphi} = h ~, \quad
f^{\varphi, ~\opid}_{\varphi\varphi\varphi, ~\varone\varone} = f_{\varone\varone} ~, \quad
f^{\varphi, ~\opid}_{\varphi\varphi\varphi, ~\varone\varphi}  = f_{\varone\varphi} ~, \quad
f^{\varphi, ~\opid}_{\varphi\varphi\varphi, ~\varphi \varone} = f_{\varphi \varone} ~,\\
\label{eq:fs_init3}
f^{\varphi, ~\opid}_{\varphi\varphi\varphi, ~\varphi\varphi} = f_{\varphi\varphi}^\opid ~, \quad
f^{\varphi, ~\vareps}_{\varphi\varphi\varphi, ~\varphi\varphi} = f_{\varphi\varphi}^\vareps ~.
\end{gather}

Next we consider some of the equations implied by \eqref{O1}:
\begin{equation}
{\def\arraystretch{1.5}
\begin{array}{cccccccccccc|cc}
a & b & c & d & e & p & q & r & s & g & m & k & \text{equation} &  
\\
\hline
\varone & \varphi & \varphi & \varphi & \varphi & \varphi & \varone & \varphi & \varone & \opid & \opid & \opid
& f_{\varone\varone}^2 + h f_{\varone \varphi} f_{\varphi \varone} = 1
& \text{(a)}
\\
\varone & \varphi & \varphi & \varphi & \varphi & \varphi & \varone & \varphi & \varphi & \opid & \opid & \opid
& f_{\varone\varone} f_{\varone\varphi} = - h f_{\varone\varphi} f_{\varphi\varphi}^\opid
& \text{(b)}
\\
\varone & \varphi & \varphi & \varphi & \varphi & \varphi & \varphi & \varphi & \varphi & \vareps & \vareps & \vareps
& h (f_{\varphi\varphi}^\vareps)^2 = h^2 \zeta^3
& \text{(c)}
\\
\varphi &  \varphi &  \varphi &  \varphi &  \varphi &  \varone &  \varphi &  \varone &  \varphi &  \sigma &  \sigma &  \sigma
& 
h^2 f_{\varone\varone} = \frac{\zeta f_{\varone\varphi} f_{\varphi \varone}}{\l}
& \text{(d)}
\\
\varphi &  \varphi &  \varphi &  \varphi &  \varphi &  \varone &  \varphi &  \varphi &  \varone &  \sigma &  \sigma &  \opid
& 
f_{\varone\varone} = \frac{h f_{\varone\varphi} f_{\varphi \varone}}{\l}
& \text{(e)}
\\
\varphi &  \varphi &  \varphi &  \varphi &  \varphi &  \varone &  \varphi &  \varphi &  \varphi &  \sigma &  \sigma &  \opid
& 
\frac{f_{\varone \varphi} f_{\varphi\varphi}^\vareps}{\zeta^3 \l^2} =
\frac{f_{\varone\varphi} f_{\varphi\varphi}^\opid}{\l}
\left( h - \frac{\zeta}{\l} \right)
& \text{(f)}
\end{array}}
\label{eq:O1eq}
\end{equation}
Let us denote $z= h f_{\varone \varphi} f_{\varphi \varone}$.
Substituting (\ref{eq:O1eq}\,e) into (\ref{eq:O1eq}\,a) yields $z^2 + \l^2 z = \l^2$.
Substituting \eqref{eq:lambda_ito_zeta} gives $z^2 + 2z = 2$, which implies
\begin{equation}
\label{eq:z_sols}
z = -1 + \d \sqrt{3}~, \quad \d=\pm 1 ~.
\end{equation}
In particular, $h$, $f_{\varone \varphi}$, $f_{\varphi \varone}$ are all non-zero.
Other constants can then also be expressed in terms of $z$ as follows:
\begin{equation}
\label{eq:fs_ito_z}
f_{\varone\varone} \overset{\text{(\ref{eq:O1eq}\,e)}}= \frac{z}{\l}~, \quad
f_{\varphi\varphi}^\opid \overset{\text{(\ref{eq:O1eq}\,b)}}= -\frac{z}{\l h}~, \quad
f_{\varphi\varphi}^\vareps \overset{\text{(\ref{eq:O1eq}\,f)}}= \frac{\zeta^3}{h} \left( \frac{\zeta}{\l} - h \right) \cdot z ~.
\end{equation}
In particular, $f_{\varone\varone}$ and $f_{\varphi\varphi}^\opid$ are non-zero.
Substituting \eqref{eq:fs_ito_z} into (\ref{eq:O1eq}\,d) yields the relation $\zeta = h^3$.
Using this, one can conclude that also $f_{\varphi\varphi}^\vareps \neq 0$.

\medskip

Conditions \eqref{O2} and \eqref{O3} are equivalent to expressing 
the constants $g^{a, ~i}_{bcd, ~pq}$, $a,b,c,d,p,q \in B$, $i \in I$ in terms of the ones in \eqref{eq:fs_init1} and \eqref{eq:fs_init3}:
\begin{gather}
    \nonumber
g^{a, ~i}_{bcd, ~pq} = N_{{}_a t_{pd} ~ {}_p t_{bc}}^i N_{{}_a t_{bq} ~ {}_q t_{cd}}^i, \quad
\text{if $\varone\in\{b,c,d\}$} ~,\\
    \nonumber
g^{\varone, ~\sigma}_{\varphi\varphi\varphi, ~\varphi\varphi} = \frac{1}{h} ~, \quad
g^{\varphi, ~\opid}_{\varphi\varphi\varphi, ~\varone\varone} = -h f_{\varphi\varphi}^\opid ~, \quad
g^{\varphi, ~\opid}_{\varphi\varphi\varphi, ~\varone\varphi} = h f_{\varone\varphi}, \quad
g^{\varphi, ~\opid}_{\varphi\varphi\varphi, ~\varphi\varone} = h f_{\varphi\varone} ~,\\
\label{eq:gs_ito_fs3}
g^{\varphi, ~\opid}_{\varphi\varphi\varphi, ~\varphi\varphi} = -h f_{\varone\varone} ~, \quad
g^{\varphi, ~\opid}_{\varphi\varphi\varphi, ~\varphi\varphi} = \frac{1}{f_{\varphi\varphi}^\vareps} ~,
\end{gather}
the rest of them being automatically zero.

\medskip

Among the equations given by the condition \eqref{O8} only two are distinct, namely
\begin{equation}
\label{eq:O8eqs}
\psi_\varone^4 + \psi_\varphi^4 = \frac{\psi_\varone^2}{\phi^2}~, \quad
2 \psi_\varone^2 \psi_\varphi^2 + \epsilon\l \psi_\varphi^2 = \frac{\psi_\varphi^2}{\phi^2} ~.
\end{equation}
The solutions to \eqref{O8} therefore are
\begin{equation}
\label{eq:psi_sols}
\psi_\varone^2 = \frac{1}{2\phi^2}\left( 1 - \frac{\nu\epsilon\l}{\sqrt{6}}  \right), \quad
\psi_\varphi^2 = \frac{\nu}{\phi^2 \sqrt{6}}, \qquad \nu = \pm 1 ~.
\end{equation}

\medskip

At this point the solutions are collected in \eqref{eq:fs_ito_z}, \eqref{eq:z_sols} and \eqref{eq:psi_sols} and para\-metrised by $(h,\d,\nu)$, where $h^{24} = -1$ (implied by \eqref{eq:I_zeta_epsilon}) and $\d, \nu\in\{1, -1\}$.
Plugging them into \eqref{O4}-\eqref{O7} and into the remaining equations implied by \eqref{O1}, one finds\footnote{\label{fn:CAS}
{} For most of these computations we used the computer algebra system Mathematica.} 
that $h$ must be a primitive 48th root of unity and that $\d$, $\nu$ are determined by $h$ and $\epsilon$ as follows, writing $h = \exp\left({\frac{\pi n i}{24}}\right)$,
\begin{center}
\begin{tabular}{|c|c|c|c|c|c|c|c|c|c|c|c|c|c|c|c|c|}
\hline
$n$   & 1 & 5 & 7 & 11 & 13 & 17 & 19 & 23 & 25 & 29 & 31 & 35 & 37 & 41 & 43 & 47\\
\hline
$\d$  & $-$ & + & + & $-$  & $-$  & +  & +  & $-$  & $-$  & +  & +  & $-$  & $-$  & +  & +  & $-$  \\
\hline
$\epsilon \nu$ & $-$ & $-$ & + & +  & +  & +  & $-$  & $-$  & $-$  & $-$  & +  & +  & +  & +  & $-$  & $-$\\
\hline
\end{tabular}
\end{center}
The signs can be expressed explicitly in terms of $h$ and $\epsilon$ as
\begin{equation}
\d  = \frac{2h^4 - h^{12}}{\sqrt{3}} ~, \quad
\nu = \frac{h^6 + h^{-6}}{\sqrt{2}} \d \epsilon ~.
\end{equation}
With the expressions derived above one can now verify that \eqref{eq:Fib-sol-3} indeed gives all solutions to \eqref{O1}--\eqref{O8}.

\medskip

Different choices for $\phi^2$ are trivially related by a rescaling as in \eqref{eq:rescaled-orbifold}.

It remains to show that different choices of $f_{\varone\varphi}$, $f_{\varphi\varone}$ such that
$h f_{\varone\varphi} f_{\varphi\varone} = \l f_{\varone\varone}$ give orbifold data that are related by $T$-compatible isomorphisms. To see this, take all constants ${}_a \l_{bc}$ in \eqref{eq:rho-direct-sum} (which correspond to a non-zero ${}_a t_{bc}$) to be equal to one, except for ${}_\varphi \l_{\varphi\varphi}$. Using \eqref{eq:f_ftilde_rels} to compute the effect on the $f$-coefficients shows that all but $f_{\varone\varphi}$ and $f_{\varphi\varone}$ are invariant, and that the latter two get multiplied by $({}_\varphi \l_{\varphi\varphi})^2$ and $({}_\varphi \l_{\varphi\varphi})^{-2}$, respectively.
\end{proof}

Let us denote the orbifold datum obtained in Theorem~\ref{thm:main-detail} by $\opA_{h,\epsilon}$. 
By Proposition~\ref{prop:T-com-rescale-equivcat} (and Theorem~\ref{thm:main-detail}), different choices for $\phi^2, f_{\varphi\varphi\varphi,~\varone\varphi}^{\varphi,~\opid} \in \opC^\times$ lead to equivalent ribbon fusion categories $(\mcI_{h^3,\epsilon})_{\opA_{h,\epsilon}}$. Hence it make sense not to include these choices into the notation for the orbifold datum $\opA_{h,\epsilon}$.

Starting from the data in \eqref{eq:Fib-sol-1}--\eqref{eq:Fib-sol-3}, by direct computation${}^{\ref{fn:CAS}}$ from \eqref{eq:dimCAAA} and Proposition~\ref{prop:count-simple},  and by Theorem \ref{thm:CA-is-modular} one obtains:

\begin{prp}\label{prop:C_A-properties}
The orbifold datum $\opA_{h,\epsilon}$ is simple. The category $(\mcI_{h^3,\epsilon})_{\opA_{h,\epsilon}}$ is a modular fusion category with 11 simple objects and has global dimension
\begin{equation}\label{eq:Fib-orb-global}
\Dim\!	\big( (\mcI_{h^3,\epsilon})_{\opA_{h,\epsilon}} \big)
~=~ 24\, \big(h^2+h^{-2}\big)^{-2} \ .
\end{equation}
\end{prp}

\begin{rem}\label{rem:which-h}
In Remark~\ref{rem:which-zeta}\,(2) we noted that extending $\mcC(sl(2),10)$ by the commutative algebra $A=\underline{0} \oplus \underline{6}$
(i.e.\ passing to the category of local modules) results in $\mcI_{\zeta,\epsilon}$ with $\zeta=\exp(\pi i \frac38)$ and $\epsilon=-1$. In \cite[Sec.\,3.4]{CRS3} it was explained how extending by a commutative algebra can be understood as a generalised orbifold, and in \cite{Mu-prep} it will be shown
that this procedure can in turn be inverted by an appropriate orbifold datum. In the $\mcC(sl(2),10)$ example, this will be achieved by one of the $\opA_{h,\epsilon=-1}$ discussed here. In fact, we can already deduce the relevant value of $h$. The condition $h^3=\zeta$ shows that $h=\exp(\pi i \frac{m}{24})$ with $m \in \{3,19,35\}$. Of these, only $m \in \{19,35\}$ gives a primitive $48^\text{th}$ roots of unity. The global dimension of $\mcC(sl(2),k)$ is $\frac{k+2}2 \big(\sin \frac{\pi}{k+2}\big)^{-2}$ and only for $m=19$ does \eqref{eq:Fib-orb-global} give the global dimension $89.5..$ of $\mcC(sl(2),10)$.
\end{rem}

\section{Analysing $\mcC_\opA$ via adjunctions}\label{sec:adjoint}

Let $\opA = (A, T, \a, \abar, \psi, \phi)$ be an orbifold datum in a modular fusion category $\mcC$.
Computing the simple objects and tensor products of $\mcC_\opA$ requires analysing the possible $T$-crossings, which can be complicated.
In this section we will illustrate that one can obtain some of that information already from looking only at the underlying bimodules.

\medskip

The forgetful functor $U : \mcC_\opA \to \mcACA$, $(M,\tau^M_i) \mapsto M$ has a biadjoint $P: \mcACA \to \mcC_\opA$, see  \cite[Rem.\,3.12]{Mulevicius:2020bat}. Explicitly, $P$ is given by
\begin{equation}
\label{eq:pipe_functor}
P(M)= \left( \im \pic[2.0]{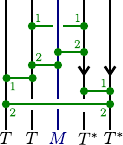}, ~\tau^{P(M)}_i \right) ~,
\end{equation}
where we do not spell out the $T$-crossings of $P(M)$ as they will not be relevant in this section. 
The horizontal lines in the above diagram are all labelled by $A$ and connected to $T$ by the first or second (co)action as indicated. One can verify that the corresponding endomorphism of $T\otimes T \otimes M \otimes T^* \otimes T^*$ is an idempotent. 

For a (multi-)fusion category $\mathcal{F}$, we denote by $\Irr\mathcal{F}$ be a set of representatives of the isomorphism classes of simple objects. In particular, $I = \Irr\,\mcC$.
    
For simple bimodules $\mu,\nu \in\Irr_{\mcACA}$
we define the constants
\begin{equation}
X_{\mu\nu} 
=
\dim \mcC_\opA(P(\mu),P(\nu))
= 
\sum_{\Lambda\in\Irr_{\mcC_\opA}} \dim \mcC_\opA(P(\mu),\Lambda) \dim \mcC_\opA(\Lambda, P(\nu)) 
~.
\end{equation}
Applying the adjunction, this can be rewritten as (we do not write out the forgetful functor)
\begin{equation}\label{eq:Xmn-via-bimodules}
X_{\mu\nu} 
=
\dim \mcACA(P(\mu),\nu)
= 
\sum_{\Lambda\in\Irr_{\mcC_\opA}} \dim \mcACA(\mu,\Lambda) \dim \mcACA(\Lambda, \nu) 
~.
\end{equation}
In particular, the first equality shows that $X_{\mu\nu}$ describes how the underlying bimodule of $P(\mu)$ decomposes into simple bimodules,
\begin{equation}
P(\mu) \stackrel{\mcACA}{\cong} \bigoplus_{\nu\in\Irr_{\mcACA}} X_{\mu\nu} ~ \nu ~.
\end{equation}
Every simple object of $\mcC_\opA$ appears as direct summand of some $P(\mu)$ (use the unit or counit of the adjunction to see this).
Thus one can attempt to find the simple objects of $\mcC_\opA$ by using the above information about Hom spaces to decompose all the $P(\mu)$.

Under assumptions \eqref{eq:assumption1}--\eqref{eq:assumption3}, firstly, the simple objects of $\mcACA$ are of the form ${}_a x_b$, where $x \in I$ and $a,b\in B$, and the left and right actions of $A$ are by the unitors of $\opid_a$ and $\opid_b$ respectively.
Secondly, by specialising \eqref{eq:pipe_functor} one gets
\begin{equation}
P({}_a x_b) = \bigoplus_{p,q,r,s,v\in B} {}_p t_{rv} \otimes {}_r t_{ua} \otimes x
\otimes {}_s t^*_{ub} \otimes {}_q t^*_{sv} ~ .
\end{equation}
Applying the fusion rules of $\mcI_{\zeta, \epsilon}$ gives
\begin{align}
X_{{}_a x_b, ~ {}_p y_q}
&= \dim \mcACA\left( ~ P({}_a x_b), ~ {}_p y_q ~ \right)
\nonumber \\
&= \sum_{i,j,k\in I} \sum_{r,s,u,v\in B}
N^{i}_{{}_p t_{rv}, ~ {}_r t_{ua}} ~
N^{j}_{i, x} ~
N^{k}_{j, ~ {}_s t_{ub}^*} ~
N^{y}_{k, ~ {}_q t_{sv}^*} ~.
\end{align}

\medskip

Let us now specialise these general considerations to the concrete example of the orbifold datum
$\opA = \opA_{h,\epsilon}$ in $\mcC = \mcI_{h^3,\epsilon}$ as given in Section~\ref{sec:orbdata}. 
Note that the precise values of $h$ and $\epsilon$ are immaterial as we only need the fusion rules of $\mcI_{\zeta,\epsilon}$ and the expressions for $A$ and $T$ in \eqref{eq:Fib-sol-1}.
The simple $A$-$A$-bimodules are
\begin{equation}
\underbrace{
{}_\varone \opid_\varone ~,~
{}_\varphi \opid_\varphi ~,~
{}_\varone \vareps_\varone ~,~
{}_\varphi \vareps_\varphi ~,~
{}_\varone \sigma_\varphi ~,~
{}_\varphi \sigma_\varone }_{\text{grade 0}} \quad,\quad
\underbrace{{}_\varone \opid_\varphi ~,~
{}_\varphi \opid_\varone ~,~
{}_\varone \vareps_\varphi ~,~
{}_\varphi \vareps_\varone ~,~
{}_\varone \sigma_\varone ~,~
{}_\varphi \sigma_\varphi}_{\text{grade 1}}  ~.
\end{equation}
The grading indicated above is respected by the tensor product $\otimes_A$, and it turns out that the matrix $X$ is block-diagonal with respect to this grading.\footnote{
   Actually, $\mcACA$ is graded by $\opZ_2 \times \opZ_2$, but only the indicated $\opZ_2$ is respected by $X$.} 
Indeed, a straightforward computation shows that, in the above ordering of the simple bimodules,
\begin{equation}
\label{eq:Ising_Xmat}
\left( X_{{}_a x_b, ~ {}_p y_q} \right) =
\left(
{\small
\begin{array}{cccccc}
2 & 3 & 0 & 1 & 1 & 1\\
3 & 8 & 1 & 6 & 5 & 5\\
0 & 1 & 2 & 3 & 1 & 1\\
1 & 6 & 3 & 8 & 5 & 5\\
1 & 5 & 1 & 5 & 4 & 4\\
1 & 5 & 1 & 5 & 4 & 4
\end{array}}
\right)
\oplus
\left(
{\small
\begin{array}{cccccc}
3 & 3 & 1 & 1 & 1 & 5\\
3 & 3 & 1 & 1 & 1 & 5\\
1 & 1 & 3 & 3 & 1 & 5\\
1 & 1 & 3 & 3 & 1 & 5\\
1 & 1 & 1 & 1 & 2 & 4\\
5 & 5 & 5 & 5 & 4 & 14
\end{array}}
\right) ~.
\end{equation}

Let us first focus on the grade-$0$ block of $X$. 
We already know a simple object of $\mcC_\opA$, namely the tensor unit $A = {}_\varone \opid_\varone \oplus {}_\varphi \opid_\varphi$. Hence, we may as well decompose each $P(\mu)$ as
\begin{equation}
P(\mu) \cong \underbrace{\mcC_\opA(P(\mu),A)}_{=\mcACA(\mu, A)} A \oplus P^{(1)}(\mu) \ ,
\end{equation}
where now $P^{(1)}(\mu)$ does no longer contain $A$ as a direct summand. Equivalently,
\begin{equation}
P^{(1)}(\mu) \cong \bigoplus_{\Lambda\in\Irr_{\mcC_\opA}, \Lambda \neq A}  \mcACA(\mu,\Lambda) \, \Lambda ~.
\end{equation}
In terms of the  $P^{(1)}(\mu)$ we can define a new matrix $X^{(1)}_{\mu\nu} = \dim \mcC_\opA(P^{(1)}(\mu),P^{(1)}(\nu))$, which can be written as
\begin{align}
X^{(1)}_{\mu\nu}
&= 
\sum_{\Lambda\in\Irr_{\mcC_\opA}, \Lambda \neq A} \dim \mcACA(\mu,\Lambda) \dim \mcACA(\Lambda, \nu)
\nonumber \\
&= 
X_{\mu\nu} - 
\dim \mcACA(\mu,A) \dim \mcACA(A, \nu)
~.
\end{align}
Explicitly, the grade-$0$ summand $X^{(1)}\big|_{\text{grade 0}}$ is given by
\begin{equation}
\left(
{\small
\begin{array}{cccccc}
2 & 3 & 0 & 1 & 1 & 1\\
3 & 8 & 1 & 6 & 5 & 5\\
0 & 1 & 2 & 3 & 1 & 1\\
1 & 6 & 3 & 8 & 5 & 5\\
1 & 5 & 1 & 5 & 4 & 4\\
1 & 5 & 1 & 5 & 4 & 4
\end{array}}
\right) -
\left(
{\small
\begin{array}{cccccc}
1 & 1 & 0 & 0 & 0 & 0\\
1 & 1 & 0 & 0 & 0 & 0\\
0 & 0 & 0 & 0 & 0 & 0\\
0 & 0 & 0 & 0 & 0 & 0\\
0 & 0 & 0 & 0 & 0 & 0\\
0 & 0 & 0 & 0 & 0 & 0
\end{array}}
\right) =
\left(
{\small
\begin{array}{cccccc}
1 & 2 & 0 & 1 & 1 & 1\\
2 & 7 & 1 & 6 & 5 & 5\\
0 & 1 & 2 & 3 & 1 & 1\\
1 & 6 & 3 & 8 & 5 & 5\\
1 & 5 & 1 & 5 & 4 & 4\\
1 & 5 & 1 & 5 & 4 & 4
\end{array}}
\right) ~.
\end{equation}
The first diagonal entry reads $\dim \mcC_\opA(P^{(1)}({}_\varone \opid_\varone),P^{(1)}({}_\varone \opid_\varone))=1$, which means that 
$P^{(1)}({}_\varone \opid_\varone) =: \D$ is itself a simple object of $\mcC_\opA$. 

Note that we can write $X^{(1)}_{\mu\nu} = \dim \mcC_\opA(P^{(1)}(\mu),P(\nu)) = \dim \mcACA(P^{(1)}(\mu),\nu)$, and so can still read off the decomposition into simple bimodules form the first row of the above matrix:
\begin{equation}
\D = 
{}_\varone \opid_\varone ~\oplus ~ 
{}_\varphi \opid_\varphi^{\oplus 2} ~\oplus ~ 
{}_\varphi \vareps_\varphi ~\oplus ~
{}_\varone \sigma_\varphi ~\oplus ~ 
{}_\varphi \sigma_\varone ~.
\end{equation}

Iterating the above procedure, we now define $P^{(2)}$ and $X^{(2)}$ by excluding the simple objects $A$ and $\D$ from the sum. One finds that $X^{(2)}\big|_{\text{grade 0}}$ is given by
\begin{equation}
\left(
{\small
\begin{array}{cccccc}
1 & 2 & 0 & 1 & 1 & 1\\
2 & 7 & 1 & 6 & 5 & 5\\
0 & 1 & 2 & 3 & 1 & 1\\
1 & 6 & 3 & 8 & 5 & 5\\
1 & 5 & 1 & 5 & 4 & 4\\
1 & 5 & 1 & 5 & 4 & 4
\end{array}}
\right) -
\left(
{\small
\begin{array}{cccccc}
1 & 2 & 0 & 1 & 1 & 1\\
2 & 4 & 0 & 2 & 2 & 2\\
0 & 0 & 0 & 0 & 0 & 0\\
1 & 2 & 0 & 1 & 1 & 1\\
1 & 2 & 0 & 1 & 1 & 1\\
1 & 2 & 0 & 1 & 1 & 1
\end{array}}
\right) =
\left(
{\small
\begin{array}{cccccc}
0 & 0 & 0 & 0 & 0 & 0\\
0 & 3 & 1 & 4 & 3 & 3\\
0 & 1 & 2 & 3 & 1 & 1\\
0 & 4 & 3 & 7 & 4 & 4\\
0 & 3 & 1 & 4 & 3 & 3\\
0 & 3 & 1 & 4 & 3 & 3
\end{array}}
\right) ~.
\end{equation}
The third diagonal entry shows that $P^{(2)}({}_\varone \vareps_\varone)$ is a direct sum of two non-isomorphic simple objects in $\mcC_\opA$, which we denote by $E_1$ and $E_2$. The corresponding row again gives the decomposition into bimodules as
\begin{equation}
E_1 \oplus E_2 =
{}_\varphi \opid_\varphi ~\oplus ~
{}_\varone \vareps_\varone^{\oplus 2} ~\oplus ~
{}_\varphi \vareps_\varphi^{\oplus 3} ~\oplus ~
{}_\varone \sigma_\varphi ~\oplus ~
{}_\varphi \sigma_\varone ~.
\end{equation}
From \eqref{eq:Xmn-via-bimodules} -- applied to $P^{(2)}$ -- one obtains a constraint on how to distribute the bimodules between $E_1$ and $E_2$. Namely, for each $\mu$ we have
\begin{equation}
X^{(2)}_{\mu\mu} ~\ge~ 
\big( \dim \mcACA(\mu,E_1) \big)^2
+
\big( \dim \mcACA(\mu,E_2) \big)^2
~.
\end{equation}
It follows that each $E_i$ must contribute one copy of ${}_\varone \vareps_\varone$, and that no one of the $E_i$ can contain all three copies of ${}_\varphi \vareps_\varphi$. If we denote the summand that contains ${}_\varphi \opid_\varphi$ by $E_1$, the remaining possibilities are
\begin{equation}
\label{eq:E1_E2_decomps}
\begin{array}{lll}
E_1
&=
&{}_\varphi \opid_\varphi ~\oplus ~
{}_\varone \vareps_\varone ~\oplus ~
{}_\varphi \vareps_\varphi^{\oplus (1+u)} ~\oplus ~
{}_\varone \sigma_\varphi^{\oplus x} ~\oplus ~
{}_\varphi \sigma_\varone^{\oplus y}\\
E_2
&=
&{}_\varone \vareps_\varone ~\oplus ~
{}_\varphi \vareps_\varphi^{\oplus (2-u)} ~\oplus ~
{}_\varone \sigma_\varphi^{\oplus (1-x)} ~\oplus ~
{}_\varphi \sigma_\varone^{\oplus (1-y)}
\end{array}, \quad
u,x,y = 0,1 ~.
\end{equation}
If we denote by $X^{(3)}$ the matrix obtained by the sum in \eqref{eq:Xmn-via-bimodules} with simple objects $A,\D,E_1,E_2$ omitted, we find
\begin{equation}
X^{(3)}\big|_{\text{grade 0}} =
\left(
{\small
\begin{array}{cccccc}
0 & 0   & 0 & 0           & 0           & 0\\
0 & 2   & 0 & 3-u         & 3-x         & 3-y\\
0 & 0   & 0 & 0           & 0           & 0\\
0 & 3-u & 0 & 2           & 1+\d_{ux}   & 1+\d_{uy}\\
0 & 3-x & 0 & 1+\d_{ux}   & 2           & 1+\d_{xy}\\
0 & 3-y & 0 & 1+\d_{uy}   & 1+\d_{xy}   & 2
\end{array}}
\right) ~.
\end{equation}
From the second row we deduce that there are two further non-isomorphic simple objects $\Phi_1$, $\Phi_2$, such that
\begin{equation}
\Phi_1 \oplus \Phi_2 =
{}_\varphi \opid_\varphi^{\oplus 2} ~\oplus ~
{}_\varphi \vareps_\varphi^{\oplus (3-u)} ~\oplus ~
{}_\varone \sigma_\varphi^{\oplus (3-x)} ~\oplus ~
{}_\varphi \sigma_\varone^{\oplus (3-y)} ~.
\end{equation}
The only way to satisfy the bound given by the diagonal entries is to have $u=x=y=1$ and to distribute the bimodule direct summands equally between $\Phi_1$ and $\Phi_2$.
Thus $\Phi_1=\Phi_2$ as bimodules (but not as objects in $\mcC_\opA$).

At this point we have found the bimodule part of all simple objects of grade $0$.
One can repeat the procedure to obtain those of grade $1$ as well.
In this case two solutions are possible: one yielding $6$ simple objects, and the other $5$.
From Proposition~\ref{prop:C_A-properties} we already know that there are $11$ simple objects in total, which eliminates the first solution.

\medskip

In order to study the tensor products of these 11 simple objects of $\mcC_\opA$ via their underlying bimodules, it is helpful to use a matrix notation for the direct sum decomposition:
\begin{equation}
M = {}_\varone W_\varone \oplus {}_\varone X_\varphi \oplus {}_\varphi Y_\varone \oplus {}_\varphi Z_\varphi
\quad \leadsto \quad
M = \begin{pmatrix}
W & X \\
Y & Z
\end{pmatrix}
~.
\end{equation}
The tensor product $\otimes_A$ is then given by matrix multiplication. In this notation, the decomposition of the simple objects in $\mcC_\opA$ into simple bimodules is\footnote{
 {} Here we write multiplicities as multiplying by integers and ``$+$'' instead of ``$\oplus$'' for better readability, that is, we work with entries in the Grothendieck ring of $\mcC$.}
\begin{align}
&\text{grade 0}
&& \text{grade 1}
\nonumber 
\\
A &= \begin{pmatrix} \opid & 0 \\ 0 & \opid \end{pmatrix}
&
S_1 &= \begin{pmatrix} \sigma & \opid \\ \opid & 2\sigma \end{pmatrix}
\nonumber 
\\
\D &= \begin{pmatrix} \opid & \sigma \\ \sigma & 2\opid+ \vareps \end{pmatrix}
&
S_2 &= \begin{pmatrix} \sigma & \vareps \\ \vareps & 2\sigma \end{pmatrix}
\nonumber 
\\
E_1 &= \begin{pmatrix} \vareps & \sigma \\ \sigma & \opid + 2 \vareps \end{pmatrix}
&
\Psi_1 &= \begin{pmatrix} 0 & \opid \\ \opid & \sigma \end{pmatrix}
\label{eq:fib-ising-simple-as-bimodule}
\\
E_2 &= \begin{pmatrix} \vareps & 0 \\ 0 & \vareps \end{pmatrix}
&
\Psi_2 &= \begin{pmatrix} 0 & \opid+\vareps \\ \opid+\vareps & 2\sigma \end{pmatrix}
\nonumber 
\\
\Phi_1 = \Phi_2 &=  \begin{pmatrix} 0 & \sigma \\ \sigma & \opid+\vareps \end{pmatrix}
&
L &= \begin{pmatrix} 0 & \vareps \\ \vareps & \sigma \end{pmatrix}
\nonumber 
\end{align}
Thus, in this example all simple objects of $\mcC_\opA$ except for $\Phi_1$, $\Phi_2$ are uniquely characterised by their underlying bimodule. Since taking duals is compatible with the underlying bimodule, in particular all simple objects except for possibly $\Phi_{1/2}$ are self-dual.

By \cite[Sec.\,3.2]{Mulevicius:2020bat}, the underlying bimodules determine the quantum dimension of objects in $\mcC_\opA$. 
Namely, if $M \in \mcC_\opA$ and $M = \sum_{a,b \in B} {}_aM_b$ as an $A$-$A$-bimodule with ${}_aM_b \in \mcC$, then, for $a \in B$,
    \begin{equation}
        \dim_{\mcC_\opA}(M) = \sum_{b \in B} \frac{\psi_b^2}{\psi_a^2} \dim_{\mcC}({}_aM_b) \ .
    \end{equation}
In particular, it follows that the above expression is independent of the choice of $a \in B$ (for $\opA$ simple), which in itself is a non-trivial condition if one tries to understand which bimodules can appear in objects of $\mcC_\opA$.

We can therefore use the expressions in \eqref{eq:fib-ising-simple-as-bimodule} to compute the quantum dimensions of the 11 simple objects. We have done that for $h =\exp(\pi i \frac{19}{24})$ as in Remark~\ref{rem:which-h} and recovered the quantum dimensions of $\mcC(sl(2),10)$, see Table~\ref{tab:dim-ident}.

\begin{table}
\begin{center}
{
\def\arraystretch{1.5}
\begin{tabular}{c|c|c|c|c|c|c}
$\dim$ &
1 &
1.93.. &
2.73.. &
3.34.. &
3.73.. &
3.86.. 
 \\
\hline
 $\mcC(sl(2),10)$    &  
 \underline{0}, \underline{10} &
 \underline{1}, \underline{9} &
 \underline{2}, \underline{8} &
 \underline{3}, \underline{7} &
 \underline{4}, \underline{6} &
 \underline{5}
 \\
\hline
$(\mcI_{h^3,\epsilon})_{\opA_{h,\epsilon}}$    & 
$A$, $E_2$ &
$\Psi_1$, $L$ &
$\Phi_1$, $\Phi_2$ &
$S_1$, $S_2$ &
$\D$, $E_1$ &
$\Psi_2$  
\end{tabular}}
\end{center}
\caption{Simple objects of $\mcC(sl(2),10)$ and of $(\mcI_{h^3,\epsilon})_{\opA_{h,\epsilon}}$ for $h =\exp(\pi i \frac{19}{24})$, sorted by quantum dimension. For $\mcC(sl(2),10)$ the Dynkin label $\underline{0}, \underline{1},\dots, \underline{10}$ is used to denote the simple objects, and for $(\mcI_{h^3,\epsilon})_{\opA_{h,\epsilon}}$ the notation in \eqref{eq:fib-ising-simple-as-bimodule} is used.  }
\label{tab:dim-ident}
\end{table}

Direct sum decompositions in $\mcC_\opA$ cannot be uniquely identified by the underlying bimodules, not even up to the ambiguity of $\Phi_{1}$ vs.\ $\Phi_2$. For example,
\begin{equation}
\Psi_1 \otimes_A \Psi_1 \cong 
\begin{pmatrix} \opid & \sigma \\ \sigma & 2\opid + \vareps \end{pmatrix}
~~.
\end{equation}
In this case,
the right hand side could be the underlying bimodule of $\D$ or of $A \oplus \Phi_{1/2}$. (However, since $\Psi_1$ is self-dual, its tensor square in $\mcC_\opA$ has to contain the tensor unit $A$ of $\mcC_\opA$ as a direct summand, and so the second decomposition is the correct one in $\mcC_\opA$.) We verified that -- taking $\Psi_1$ has the generator -- the iterated tensor products of $\Psi_1$ are compatible with those of $\mcC(sl(2),10)$.

The quantum dimensions of simple objects are also useful in showing the following result.

\begin{prp}\label{prop:all-different}
The 32 modular fusion categories $(\mcI_{h^3,\epsilon})_{\opA_{h,\epsilon}}$ for the 16 possible values of $h$ and $\epsilon \in \{\pm 1\}$ are pairwise non-equivalent as $\opC$-linear ribbon categories.
\end{prp}

\begin{proof}
A ribbon equivalence preserves the global dimension and the anomaly. 
It is shown in \cite{Carqueville:2020} that passing to the orbifold modular fusion category preserves the anomaly. Hence the anomaly $\xi$ of $(\mcI_{h^3,\epsilon})_{\opA_{h,\epsilon}}$ is equal to that of $\mcI_{h^3,\epsilon}$ as stated in \eqref{eq:Ising-Dim-anomaly}. 

A ribbon equivalence also preserves the quantum dimension of simple objects, and from this one can verify that any equivalence must map $\Psi_1$ for one choice of $(h,\epsilon)$ to either $\Psi_1$ or $L$ for any other choice $(h',\epsilon')$.

Abbreviating $\mcD = (\mcI_{h^3,\epsilon})_{\opA_{h,\epsilon}}$, altogether we see that the triple of numbers
\begin{equation}
    \big( \,
    \Dim(\mcD) \,,\,  
    \xi(\mcD) \,,\,  
    \dim_\mcD(\Psi_1) \,
    \big)
\end{equation}
is a ribbon invariant. Note that $\dim_\mcD(\Psi_1)=\dim_\mcD(L)$, so it does not matter whether we use $\Psi_1$ or $L$.
From \eqref{eq:Fib-orb-global}, \eqref{eq:Ising-Dim-anomaly} and \eqref{eq:Fib-sol-3} we read off the explicit values to be
\begin{equation}
    \big( \,
    24\, \big(h^2+h^{-2}\big)^{-2}  \,,\,  
    \epsilon h^{-3} \,,\,  
    -\epsilon (h^{10}+h^{-10}) \,
    \big) \ .
\end{equation}
It is straightforward to check that this distinguishes all 32 possibilities.
\end{proof}

\appendix

\section{Appendix}

\subsection{Labelling convention for polynomial equations}\label{app:convention}

The defining data and conditions for an orbifold datum arise from the study of generalised orbifolds of Reshetikhin-Turaev TQFTs \cite{CRS1,CRS3}. There, surface defects are labelled by algebras and line defects by appropriate multi-modules. In the pictures below, we give the defect version of the algebraic condition as we find these to be the most readable way to present them. The shading on some surfaces indicates that their orientation is opposite to that of the paper plane.

By assumption \eqref{eq:assumption2}, the algebra in question is a direct sum of copies of the tensor unit, indexed by elements of $B$. In the pictures below, we indicate for each surface the index in $B$ that we used when converting the corresponding algebraic equations (O1)--(O8) in \cite[Def.\,2.2]{Mulevicius:2020bat} into the polynomial equations in Table~\ref{tab:orb_eqs}. For (O1) this is explained in slightly more detail in Appendix~\ref{app:pentagon}.

\begin{align*}
& 
\raisebox{7em}{\text{O1\,:}}\hspace{-2em}
\pic[0.85]{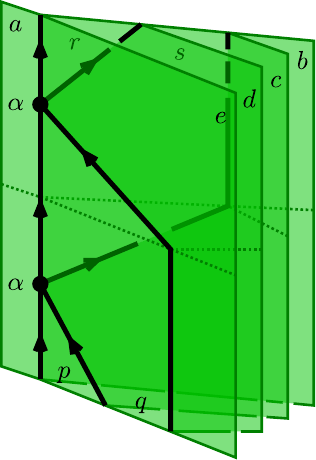} = \pic[0.85]{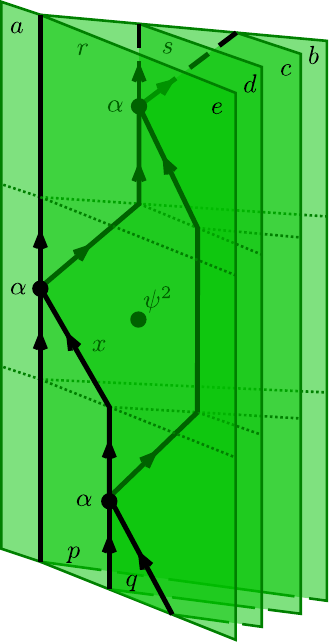}
\\
&
\raisebox{6.5em}{\text{O2\,:}}\hspace{-2em}
\pic[0.85]{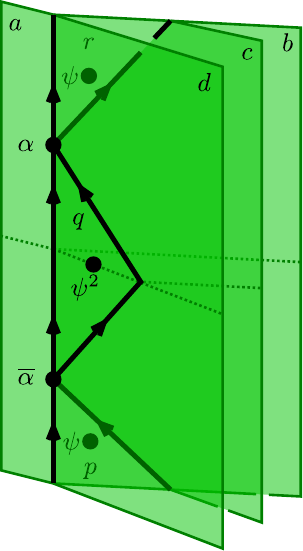} = \pic[0.85]{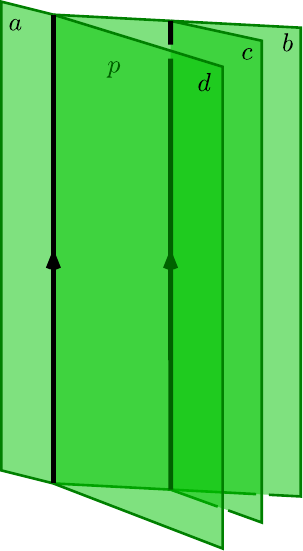}
&&	
\raisebox{6.5em}{\text{O3\,:}}\hspace{-2em}
\pic[0.85]{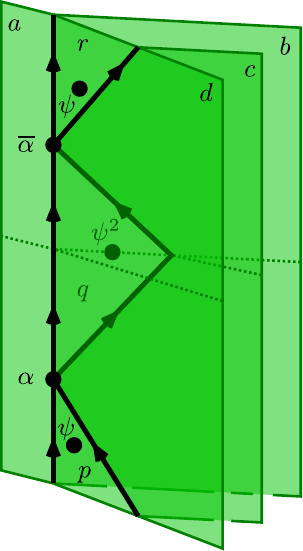} = \pic[0.85]{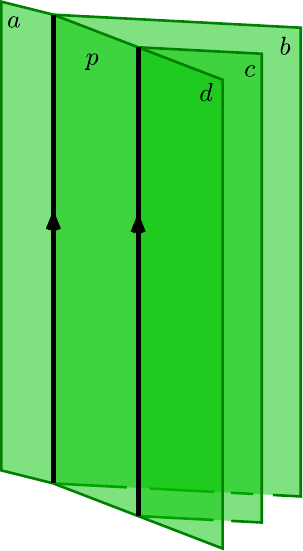}
\\
&	
\raisebox{6.5em}{\text{O4\,:}}\hspace{-2em}
\pic[0.85]{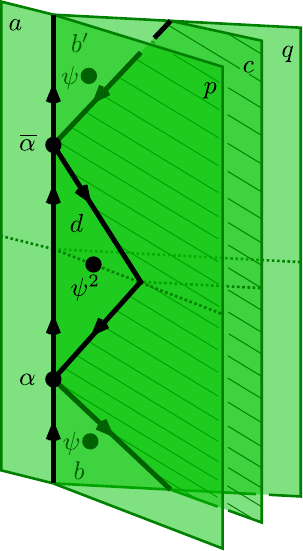} = \pic[0.85]{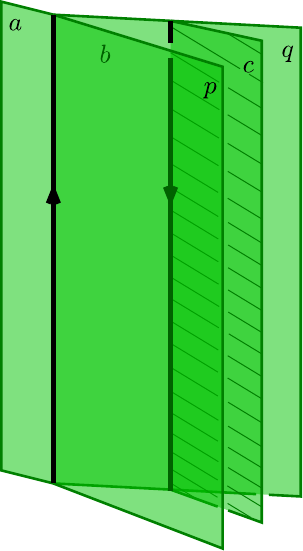}
&&	
\raisebox{6.5em}{\text{O5\,:}}\hspace{-2em}
\pic[0.85]{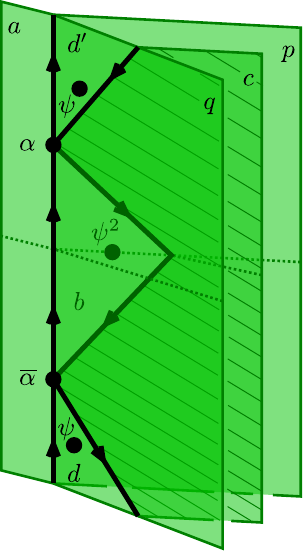} = \pic[0.85]{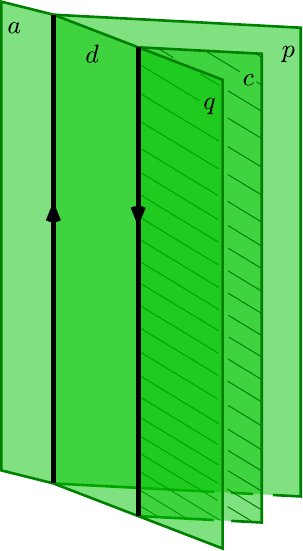}
\\
&	
\raisebox{6.5em}{\text{O6\,:}}\hspace{-2em}
	\pic[0.85]{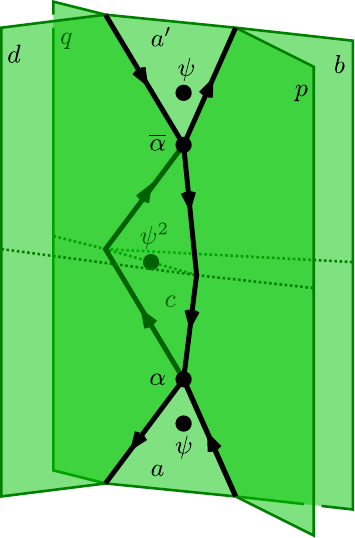} = \pic[0.85]{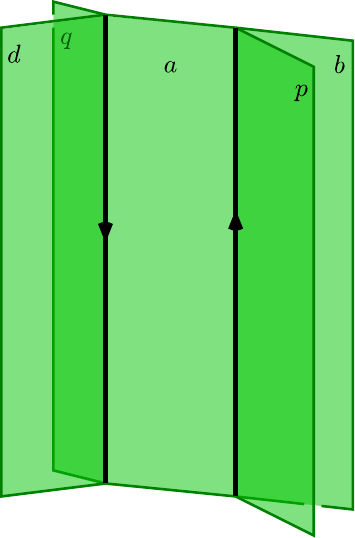}
&~~&	
\raisebox{6.5em}{\text{O7\,:}}\hspace{-2em}
	\pic[0.85]{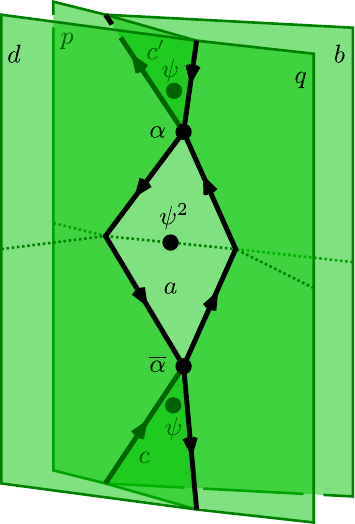} = \pic[0.85]{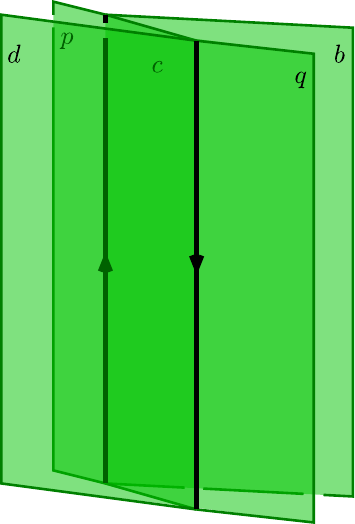}
\end{align*}
\begin{align*}
&	
\raisebox{4em}{\text{O8\,:}}\hspace{-2em}
	\pic[0.85]{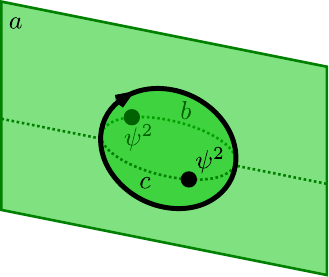} = \pic[0.85]{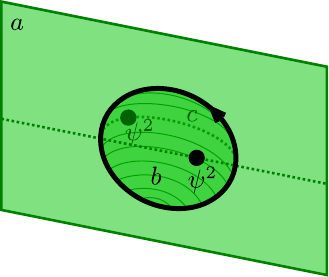} = \pic[0.85]{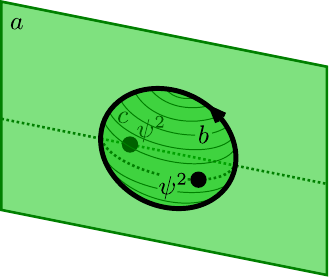} = 
	\phi^{-2} \cdot 
	\pic[0.85]{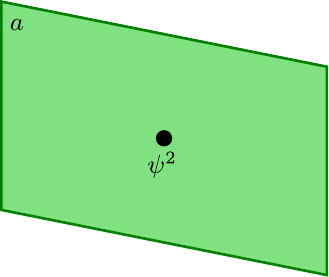}
\end{align*}

\subsection{Deriving the pentagon condition}
\label{app:pentagon}
The pentagon condition (identity (O1) in \cite[Def.\,2.2]{Mulevicius:2020bat}) for an orbifold datum $\opA = (A, T, \a, \abar, \psi, \phi)$ in a modular fusion category $\mcC$ is
\begin{equation}
\label{eq:O1}
\pic[2.0]{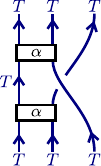} = \pic[2.0]{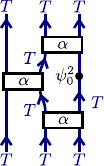} ~.
\end{equation}
This algebraic condition is obtained directly from the defect presentation O1 in the previous appendix.
Under assumptions \eqref{eq:assumption1}--\eqref{eq:assumption3} it is equivalent to the collection of identities
\begin{equation}
\label{eq:O1_assump}
\pic[2.0]{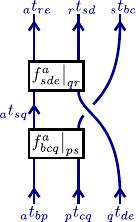} = \sum_{x\in B} \pic[2.0]{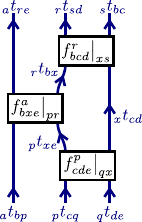},\quad
a,b,c,d,e,p,q,r,s\in B ~,
\end{equation}
where we used \eqref{eq:alpha_ito_f} and cancelled the appearances of $\psi$ on both sides.
Notice that the indexing matches the defect presentation O1 in the previous appendix. For example, on the left hand side of \eqref{eq:O1_assump}, the object ${}_at_{bp}$ at the bottom left corresponds in the defect presentation O1 to the bottom left line defect, which is connected to a surface defect labelled $a$ on the left, and to two surface defects labelled $b$ and $p$ on the right.

Let us decompose the source and target of the morphism in \eqref{eq:O1_assump} into simple objects by composing both sides with
\begin{equation}
\label{eq:O1_decompose}
\pic[2.0]{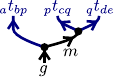}, \quad \pic[2.0]{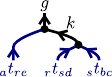}, \quad
g,k,m\in I ~.
\end{equation}
On the left hand side one gets:
\begin{align}
&\pic[2.0]{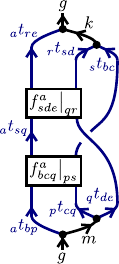} =
\sum_{i,j\in I} f^{a, ~i}_{sde, ~ qr} f^{a, ~j}_{bcq, ~ ps} \pic[2.0]{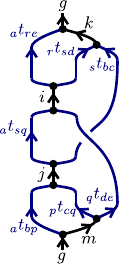}\\
& =\sum_{i,j\in I}
f^{a, ~i}_{sde, ~ qr} ~
f^{a, ~j}_{bcq, ~ ps} ~
F^{\left( {}_a t_{re} ~ {}_r t_{sd} ~ {}_s t_{bc} \right) g}_{k \, i} ~
G^{\left( {}_a t_{bp} ~ {}_p t_{cq} ~ {}_q t_{de} \right) g}_{j \, m}
\pic[2.0]{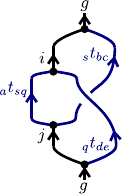} ~.
\label{eq:app-pentagon-lhs}
\end{align}
The remaining diagram can be quickly evaluated as follows:
\begin{align}
&\pic[2.0]{A1_O1_lhs_calc3.pdf} =
\pic[2.0]{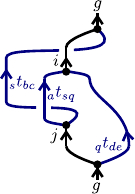} =
R^{\left( {}_a t_{sq} ~ {}_s t_{bc} \right) j} ~
R^{-\left( {}_s t_{bc} ~ i \right) g}
\pic[2.0]{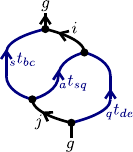}
\nonumber \\
&=
R^{\left( {}_a t_{sq} ~ {}_s t_{bc} \right) j} ~
R^{-\left( {}_s t_{bc} ~ i \right) g} ~
F^{\left( {}_s t_{bc} ~ {}_a t_{sq} ~ {}_q t_{de} \right) g}_{i \, j} \cdot \id_g ~.
\end{align}
Similarly we compute the right hand side of \eqref{eq:O1_assump} upon composing with the morphisms in \eqref{eq:O1_decompose}:
\begin{align}
&\sum_{x\in B}\pic[2.0]{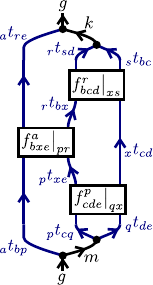} =
\sum_{x\in B, ~l\in I}
f^{p, ~m}_{cde, ~qx} ~
f^{a, ~l}_{bxe, ~pr} ~
f^{r, ~k}_{bcd, ~xs}
\pic[2.0]{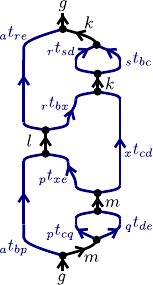}
\nonumber \\
&=\sum_{x\in B, ~l\in I}
f^{p, ~m}_{cde, ~qx} ~
f^{a, ~l}_{bxe, ~pr} ~
f^{r, ~k}_{bcd, ~xs} ~
N_{{}_p t_{cq} ~ {}_q t_{de}}^m ~
N_{{}_r t_{sd} ~ {}_s t_{bc}}^k
\pic[2.0]{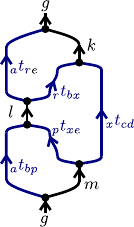}
\nonumber \\
&=\sum_{x\in B, ~l\in I}
f^{p, ~m}_{cde, ~qx} ~
f^{a, ~l}_{bxe, ~pr} ~
f^{r, ~k}_{bcd, ~xs} ~
F^{\left( {}_a t_{re} ~ {}_r t_{bx} ~ {}_x t_{cd} \right) g}_{k \, l} ~
G^{\left( {}_a t_{bp} ~ {}_p t_{xe} ~ {}_x t_{cd} \right) g}_{l \, m} 
\cdot \id_g ~.
\label{eq:app-pentagon-rhs}
\end{align}
Comparing coefficients in \eqref{eq:app-pentagon-lhs} and \eqref{eq:app-pentagon-rhs} gives condition~\eqref{O1} in Table~\ref{tab:orb_eqs}.

\subsection{Evaluating the $T^3$-invariant}
\label{app:T3_invariant}
Let $\mcC$ be a modular fusion category and $\opA = (A,T,\a,\abar,\psi,\phi)$ a simple orbifold datum in $\mcC$, satisfying assumptions \eqref{eq:assumption1}--\eqref{eq:assumption3}.
To compute the invariant $Z^\mathrm{orb,\opA}_\mcC(T^3)$ of the $3$-torus $T^3$ one must first pick an \textit{$\opA$-decorated skeleton} of it (see \cite{Carqueville:2020} for details), i.e.\ a stratification $S$ of $T^3$, such that all $3$-strata are homeomorphic to $\opR^3$, $2$-strata are labelled by $A$, $1$-strata are labelled by $T$ and $0$-strata are labelled by $\a$, $\abar$ depending on the orientation. 
Our choice is going to be the following:
\begin{equation}
\pic[1.75]{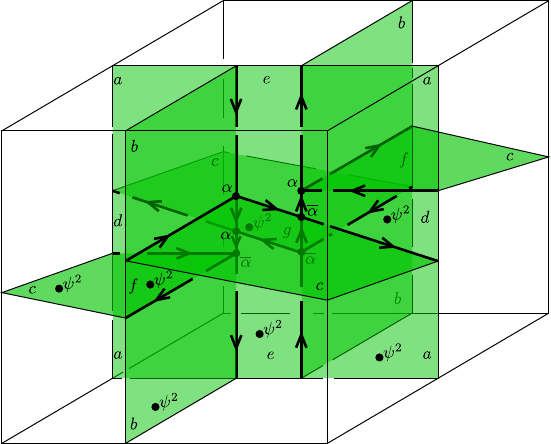} ~.
\end{equation}
Here $T^3$ is depicted as a cube with the opposite sides identified, all of the $2$-strata have the paper plane orientation.  
The label on each 2-stratum matches the index of the summand $\opid$ in the direct sum decomposition of the corresponding copy of $A$ as used below.

\medskip

Next one converts the stratification to a $\mcC$-coloured ribbon graph embedded in $T^3$ and evaluates it with the Reshetikhin-Turaev TQFT $Z_\mcC$ as described in \cite{CRS2,CRS3}.
Using the expressions \eqref{eq:alpha_abcd}, \eqref{eq:alpha_ito_f}, \eqref{eq:abar_ito_g}, \eqref{eq:fg_scalars} for $\a$ and $\abar$ one gets:
\begin{align} \nonumber
Z^\mathrm{orb,\opA}_\mcC & (T^3) =\\ \nonumber
&\phi^2 \cdot \sum_{\substack{a,b,c,d,e,f,g\in B \\ r,s,t,u,v,w\in I}}
\frac{\psi_a^2 \psi_c^2 \psi_e^2 \psi_f^2}{\psi_b^2 \psi_d^2 \psi_g^2} ~
f^{e, ~r}_{afc, ~bg} ~ f^{e, ~s}_{caf, ~gd} ~ f^{e, ~u}_{fca, ~db} ~
g^{e, ~t}_{acf, ~db} ~ g^{e, ~v}_{fac, ~gd} ~ g^{e, ~w}_{cfa, ~bg}\\ \label{eq:T3_calc1}
& ~\cdot Z_\mcC ( T^3_{\operatorname{def}} ) ~.
\end{align}
where
\begin{equation}
\label{eq:T3def}
T^3_{\operatorname{def}} := \pic[1.75]{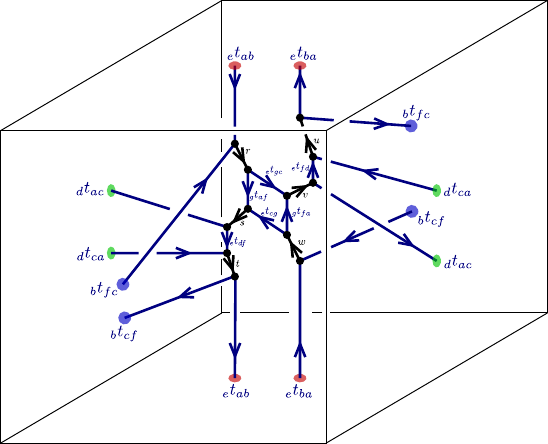} ~.
\end{equation}

\medskip

To evaluate the invariant of the torus \eqref{eq:T3def}, let us introduce the scalars $L^{bcf, ~wt}_{eda, ~ur}\big|_x$, where $a,b,c,d,e \in B$ and $x,u,r,w,t\in I$, 
such that
\begin{equation}
\label{eq:L_definition}
\pic[2.0]{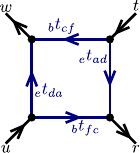} =
\sum_{x\in I} L^{bcf, ~wt}_{eda, ~ur}\big|_x \pic[2.0]{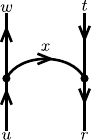} ~.
\end{equation}
We claim that the following equality holds:
\begin{equation}
\label{eq:L_formula}
L^{bcf, ~wt}_{eda, ~ur}\big|_x =
N^{w}_{{}_e t_{da} ~ {}_b t_{cf}} ~
N^{t}_{{}_e t_{ad} ~ {}_b t_{cf}} ~
\frac{\dim x}{\dim {}_b t_{fc}} ~
F^{({}_e t_{ad} ~ {}_b t_{cf} ~ x)r}_{{}_b t_{fc} ~ t} ~
G^{({}_e t_{da} ~ {}_b t_{cf} ~ x)u}_{w ~ {}_b t_{fc}} ~.
\end{equation}
Indeed, use the identity
\begin{equation}
\pic[2.0]{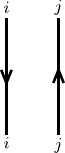} =
\sum_{k\in I} \frac{\dim k}{\dim j} \pic[2.0]{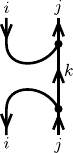}
\end{equation}
to rewrite the left hand side of \eqref{eq:L_definition} as
\begin{align}
\nonumber
& \pic[2.0]{A2_L_loop_lhs.pdf} = \sum_{x\in I} \frac{\dim x}{\dim {}_b t_{fc}}
\pic[2.0]{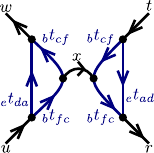}\\
&= \sum_{x\in I} \frac{\dim x}{\dim {}_b t_{fc}}
G^{({}_e t_{da} ~ {}_b t_{fc} ~ x)u}_{w ~ {}_b t_{fc}} ~
F^{({}_e t_{ad} ~ {}_b t_{cf} ~ x)r}_{{}_b t_{fc} ~ t}
\pic[2.0]{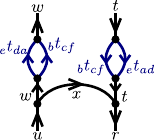} ~.
\end{align}

\medskip

\begin{figure}
\centering
\begin{subfigure}[b]{0.49\textwidth}
	\centering
	\pic[1.3]{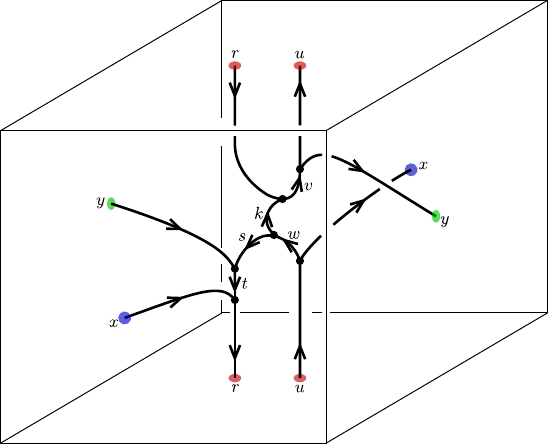}
	\caption{}
	\label{fig:T3_graph2}
\end{subfigure}
\begin{subfigure}[b]{0.49\textwidth}
	\centering
	\pic[1.3]{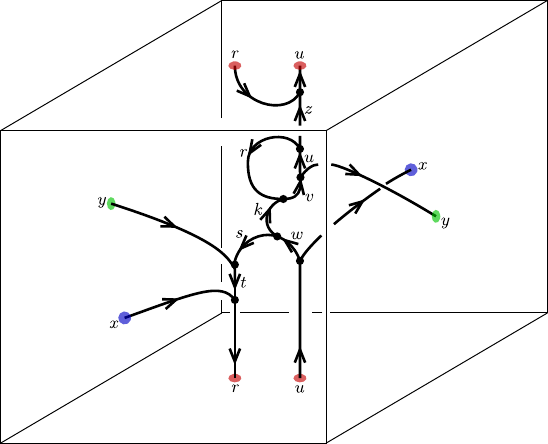}
	\caption{}
	\label{fig:T3_graph3}
\end{subfigure}\\
\begin{subfigure}[b]{0.49\textwidth}
	\centering
	\pic[1.3]{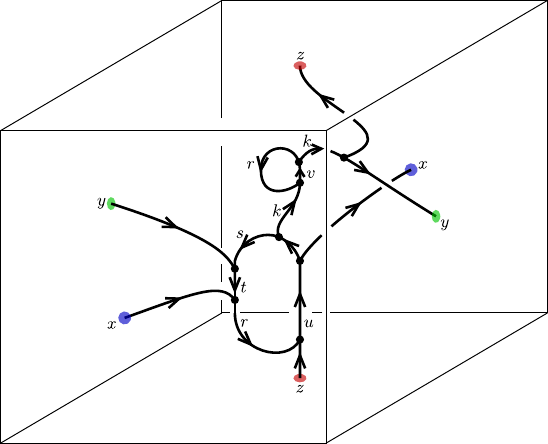}
	\caption{}
	\label{fig:T3_graph4}
\end{subfigure}
\begin{subfigure}[b]{0.49\textwidth}
	\centering
	\pic[1.3]{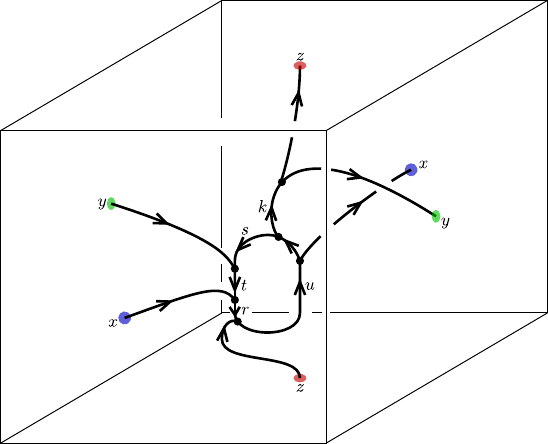}
	\caption{}
	\label{fig:T3_graph5}
\end{subfigure}\\ \pagebreak
\begin{subfigure}[b]{0.49\textwidth}
	\centering
	\pic[1.3]{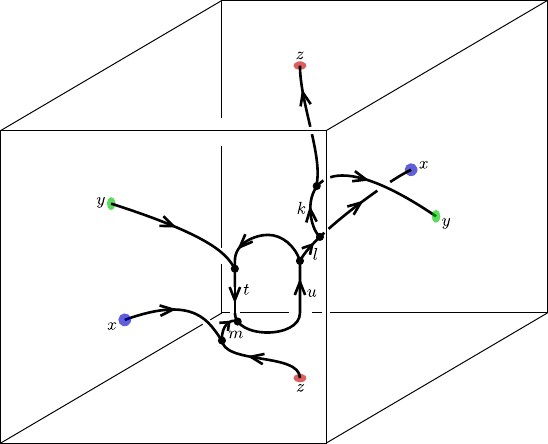}
	\caption{}
	\label{fig:T3_graph6}
\end{subfigure}
\begin{subfigure}[b]{0.49\textwidth}
	\centering
	\pic[1.3]{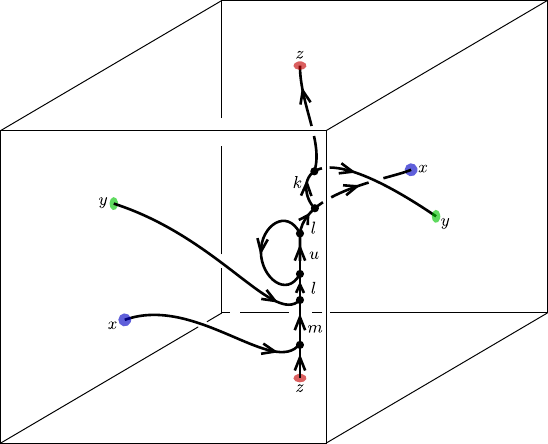}
	\caption{}
	\label{fig:T3_graph7}
\end{subfigure}
\caption{Three-torus with a series of embedded ribbon graphs as used in the calculation in \eqref{eq:T3calc2}.}
\label{fig:T3_graphs}
\end{figure}

We now focus on evaluating the invariant of the torus $T^3_{\operatorname{def}}$ as in \eqref{eq:T3def}.
Let us denote by $T^3_{\text{(a)}}$, $T^3_{\text{(b)}}$, \dots, $T^3_{\text{(f)}}$ the $3$-tori with embedded ribbon graphs as depicted in Figure~\ref{fig:T3_graphs}.
One has:
\begin{equation}
\label{eq:T3calc2}
Z_\mcC(T^3_{\operatorname{def}}) =
\sum_{x,y,k\in I}
L^{bcf, ~wt}_{eba, ~ur} \big|_x ~
L^{dca, ~us}_{efd, ~vt} \big|_y ~
L^{gaf, ~sr}_{ecg, ~wv} \big|_k \cdot
Z_\mcC(T^3_{\text{(a)}}) ~,
\end{equation}
where
\begin{align}
\nonumber
& Z_\mcC(T^3_{\text{(a)}}) =
\sum_{z\in I}
\frac{\dim z}{ \dim u} \cdot
Z_\mcC(T^3_{\text{(b)}}) ~,
&& Z_\mcC(T^3_{\text{(b)}}) =
F^{(r \, z \, y)\,v}_{k \, u} 
\cdot 
Z_\mcC(T^3_{\text{(c)}}) ~, \\
\nonumber
& Z_\mcC(T^3_{\text{(c)}}) =
\frac{\dim v}{\dim k} \cdot
Z_\mcC(T^3_{\text{(d)}}) ~,
&& Z_\mcC(T^3_{\text{(d)}}) =
\sum_{l,m\in I}
F^{(s \, k \, x)u}_{l \, w} ~
G^{(t \, x \, z)u}_{r \, m} \cdot
Z_\mcC(T^3_{\text{(e)}}) ~, \\
\nonumber
& Z_\mcC(T^3_{\text{(e)}}) =
G^{(s \, y \, m)u}_{t \, l} \cdot
Z_\mcC(T^3_{\text{(f)}}) ~,
&& Z_\mcC(T^3_{\text{(f)}}) =
\frac{\dim u}{\dim l} \cdot T_{xyz, ~ klm} ~,
\end{align}
where in the last equation we use the notation
\begin{equation}
\label{eq:Txyzklm_def}
T_{xyz, ~ klm} :=
Z_\mcC \left(
\pic[1.5]{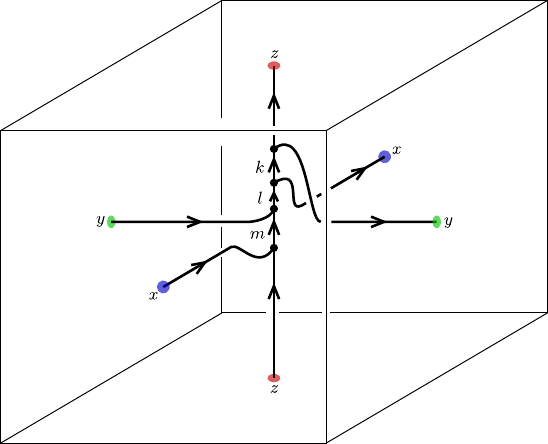}
\right) ~.
\end{equation}
Combining all of the equations in \eqref{eq:T3calc2} one already obtains \eqref{eq:orb-simple}.
It remains to derive the expression for $T_{xyz, ~klm}$.

\medskip

The invariant of $T^3 = S^1 \times S^1 \times S^1$ can also be computed as the trace of the operator invariant assigned to the cylinder $C = S^1 \times S^1 \times [0,1]$.
The same is true if $T^3$ has an embedded ribbon graph, in which case $C$ can have punctures on its boundary.
For the cylinder corresponding to the $3$-torus in the argument of $Z_\mcC$ in \eqref{eq:Txyzklm_def} we will use the graphical representation
\begin{equation}
\label{eq:cylinder_picture}
\pic[2.0]{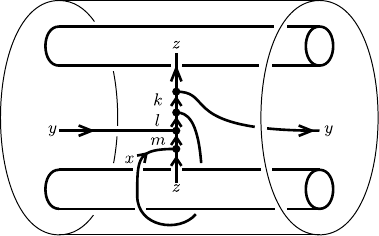}
\end{equation}
where the outer tube represents the manifold $S^2 \times S^1$ (the ends of the tube on the left and on the right are assumed to be identified and the boundary circle of a vertical slice corresponds to the point at infinity of $S^2$). The two small inner tubes represent the boundary components (seen here as obtained by removing two solid tori from $S^2 \times S^1$).
The tube at the bottom corresponds to the incoming boundary component, while the one at the top to the outgoing one.

\medskip

The vector space assigned to a $2$-torus with a single $z\in I$ labelled puncture is isomorphic to $\bigoplus_{q\in I} \mcC(q, q \otimes z)$ and its dual to $\bigoplus_{p\in I} \mcC(p \otimes z, p)$.
The image of a basis element $\l^{(qz)q}$ and evaluation with the dual basis element $\l_{(pz)p}$ are obtained by gluing the solid tori
\begin{equation*}
\pic[2.0]{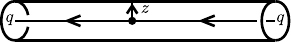}
\quad\text{and}\quad
\pic[2.0]{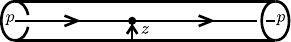}
\end{equation*}
to the incoming and outgoing boundary respectively.
One then has:
\begin{equation}
\label{eq:T3_calc3}
T_{xyz, ~klm} = \sum_{p\in I}
Z_\mcC\left( \pic[2.0]{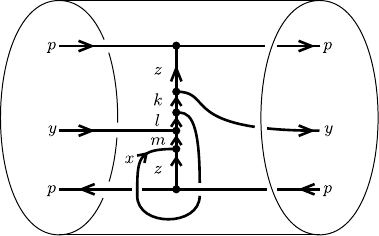} \right) ~.
\end{equation}
Now, the invariant of $S^2 \times S^1$ is equal to the trace of the operator invariant of the cylinder $S^2 \times [0,1]$.
The vector space assigned to a $2$-sphere with three punctures labelled by $p$, $y$, $p^*$ as in \eqref{eq:T3_calc3} is $\mcC(\opid,p\otimes y\otimes p^*)$ with the dual $\mcC(p\otimes y \otimes p^*, \opid)$.
Using the basis
\begin{equation}
\pic[2.0]{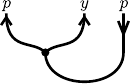} \quad\text{with dual}\quad
\frac{1}{\dim p}
\pic[2.0]{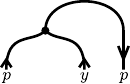}
\end{equation}
one has:
\begin{equation}
\label{eq:T3_calc4}
T_{xyz, ~klm} =
\sum_{p\in I} \frac{N_{py}^p}{\dim p} 
\underbrace{\pic[2.0]{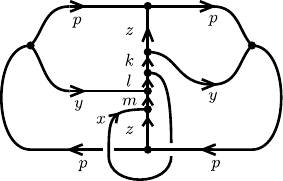}}_{=:\Gamma}
~.
\end{equation}
The scalar represented by the string diagram $\Gamma$ can be expressed as follows:
\begin{align*}
\Gamma &= 
G^{(p \, z \, y)p}_{p \, k} \cdot 	\pic[1.65]{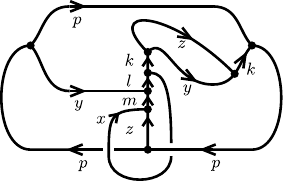}
\\
&= 
G^{(p \, z \, y)p}_{p \, k} N_{z y}^k \cdot 	\pic[1.65]{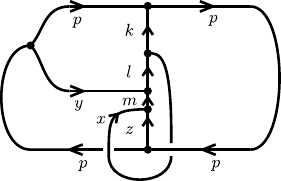}
\\
& \overset{(*)}=
G^{(p \, z \, y)p}_{p \, k}  \sum_{j\in I} G^{(p \, k \, x)j}_{p \, l} \cdot 	\pic[1.65]{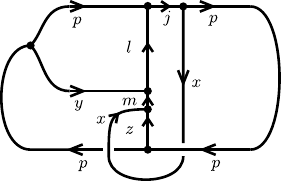}
\\
&  =
G^{(p \, z \, y)p}_{p \, k} \sum_{j\in I} G^{(p \, k \, x)j}_{p \, l}   F^{(p \, y \, m)j}_{l \, p} \cdot 	\pic[1.65]{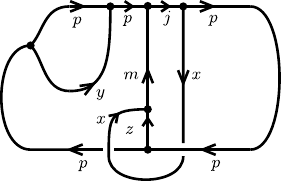}
\\
&= 
G^{(p \, z \, y)p}_{p \, k} \sum_{j\in I} G^{(p \, k \, x)j}_{p \, l}   F^{(p \, y \, m)j}_{l \, p} N_{p y}^p \cdot 
	\pic[1.65]{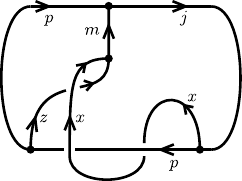}
\\
&\overset{(**)}= 
G^{(p \, z \, y)p}_{p \, k}  \sum_{j\in I} G^{(p \, k \, x)j}_{p \, l}   F^{(p \, y \, m)j}_{l \, p} R^{-(z \, x)m} \frac{\theta_j}{\theta_x \theta_p} F^{(p \, z \, x)j}_{m \, p} \cdot \dim j ~.
\end{align*}
In step $(*)$ we omitted $N_{zy}^k$ as it is implicit in $G^{(pzy)p}_{pk}$. In step $(**)$ we omitted $N_{p y}^p$ for the same reason, and we used the identities
\begin{equation}
\pic[2.0]{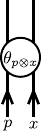} = \pic[2.0]{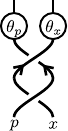}, \quad
\pic[1.75]{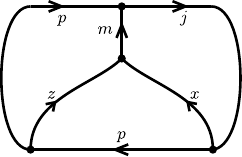} = F^{(p \, z \, x)j}_{m \, p} \cdot \dim j ~.
\end{equation}
Substituting the above expression for $\Gamma$ into \eqref{eq:T3_calc4} yields exactly the expression for $T_{xyz, ~klm}$ in Lemma \ref{lem:T3_invariant}.

\newcommand{\arxiv}[2]{\href{http://arXiv.org/abs/#1}{#2}}
\newcommand{\doi}[2]{\href{http://doi.org/#1}{#2}}

\end{document}